\documentclass[11pt]{amsart}
\usepackage{amsmath,amsfonts,amsthm,amsopn,amssymb,latexsym,soul}
\usepackage{cite}
\usepackage{color,enumitem,graphicx}
\usepackage[colorlinks=true,urlcolor=blue,
citecolor=red,linkcolor=blue,linktocpage,pdfpagelabels,
bookmarksnumbered,bookmarksopen]{hyperref}
\usepackage[english]{babel}
\usepackage[left=2.61cm,right=2.61cm,top=2.72cm,bottom=2.72cm]{geometry}
\usepackage[hyperpageref]{backref}
\usepackage[colorinlistoftodos]{todonotes}

\makeatletter
\providecommand\@dotsep{5}
\def\listtodoname{List of Todos}
\def\listoftodos{\@starttoc{tdo}\listtodoname}
\makeatother

\numberwithin{equation}{section}
\newtheorem{theorem}{Theorem}[section]
\newtheorem{lemma}[theorem]{Lemma}

\newtheorem{definition}[theorem]{Definition}
\newtheorem{proposition}[theorem]{Proposition}
\newtheorem{remark}[theorem]{Remark}
\newtheorem{corollary}[theorem]{Corollary}
\newcommand{\eps}{\varepsilon}

\newcommand{\R}{\mathbb{R}}

\newcommand{\RN}{{\mathbb{R}^N}}
\newcommand{\RT}{{\mathbb{R}^3}}

\newcommand{\de}{\partial}

 \DeclareMathOperator{\dv}{div}
 
\DeclareMathOperator{\supp}{supp} 
\renewcommand{\le}{\leqslant}
\renewcommand{\ge}{\geqslant}
\renewcommand{\a }{\alpha }

\renewcommand{\d }{\delta }
\newcommand{\vfi}{\varphi}

\newcommand{\n }{\nabla }
\newcommand{\s }{\sigma }

\renewcommand{\O}{\Omega}

\newcommand{\G}{\Gamma}

\newcommand{\I}{\mathcal{I}}

\newcommand{\X}{\mathcal{X}}
\newcommand{\N}{\mathbb{N}}

\renewcommand{\o}{\omega}

\newcommand{\D }{D^{1,2}(\RN)}
\newcommand{\irn }{\int_{\RN}}

\def\bbm[#1]{\mbo\X{\boldmath $#1$}}
\newcommand{\beq }{\begin{equation}}
\newcommand{\eeq }{\end{equation}}

\title[Electrostatic Born-Infeld equation with extended charges]{On the electrostatic Born-Infeld equation\\ with extended charges}

\author[D. Bonheure]{Denis Bonheure}
\address{D\'epartement de Math\'ematique, Universit\'e libre de Bruxelles, 
\newline\indent 
CP 214, Boulevard du Triomphe, B-1050 Bruxelles, Belgium
\newline\indent 
and INRIA - \'Equipe MEPHYSTO}
\email{denis.bonheure@ulb.ac.be}

\author[P. d'Avenia]{Pietro d'Avenia}
\address{Dipartimento di Meccanica, Matematica e Management,
Politecnico di Bari
\newline\indent
Via Orabona 4,  70125  Bari, Italy}
\email{pietro.davenia@poliba.it}

\author[A. Pomponio]{Alessio Pomponio}
\address{Dipartimento di Meccanica, Matematica e Management,
Politecnico di Bari
\newline\indent
Via Orabona 4,  70125  Bari, Italy}
\email{alessio.pomponio@poliba.it}

\thanks{D. Bonheure is supported by INRIA - Team MEPHYSTO, 
MIS F.4508.14 (FNRS), PDR T.1110.14F (FNRS) 
\& ARC AUWB-2012-12/17-ULB1- IAPAS. 
P. d'Avenia and A. Pomponio are supported by GNAMPA Project ``Analisi variazionale di modelli fisici non lineari''}

\begin{document}
\begin{abstract}
In this paper, we deal with the electrostatic Born-Infeld equation
\begin{equation}\label{eq:BI-abs}
\tag{$\mathcal{BI}$}
\left\{
\begin{array}{ll}
-\operatorname{div}\left(\displaystyle\frac{\nabla \phi}{\sqrt{1-|\nabla \phi|^2}}\right)= \rho \quad\hbox{in }\mathbb{R}^N,
\\ [4mm]
\displaystyle\lim_{|x|\to \infty}\phi(x)= 0,
\end{array}
\right.
\end{equation}
where $\rho$ is an assigned extended charge density. 
We are interested in the existence and uniqueness of the potential $\phi$ and finiteness of the energy of the electrostatic field $-\nabla \phi$. 
We first relax the problem and treat it with the direct method of the Calculus of Variations for a broad class of charge densities. Assuming $\rho$ is radially distributed, we recover the weak formulation of \eqref{eq:BI-abs} and the regularity of the solution of the Poisson equation (under the same smootheness assumptions). 
In the case of a locally bounded charge, we also recover the weak formulation without assuming any symmetry. The solution is even classical if $\rho$ is smooth. Then we analyze the case where the density $\rho$ is a superposition of point charges and discuss the results in \cite{K}.
Other models are discussed, as for instance a system arising from the coupling of the nonlinear Klein-Gordon equation with the Born-Infeld theory. 
\end{abstract}

\maketitle

\tableofcontents

\section{Introduction}

Classically, the relation between matter and electromagnetic field in the space-time $\R\times\R^3$ can be interpreted from two opposite points of view. Following a unitarian approach, we can consider the electromagnetic field as the unique physical entity and the matter is then given by the singularities of the field.
Conversely, according to a dualistic standpoint, field and particle are two different entities: even if the particles generate the field and interact with it,
they are not a part of the field.
In the dualistic approach, however, according to the original Maxwell theory of electromagnetism, there is the so called infinity problem associated with a point charge source \cite{BI,Fey}. 
More precisely, choosing suitably the physical constants, the classical Maxwell equations for the electrostatic case in the vacuum are
\begin{align}
\nabla \times \mathbf{E}&=0, \label{eq:1}
\\
\dv\mathbf{E}&=\rho, \label{eq:2}
\end{align}
where $\bf E$ is the electric field and $\rho$ is the charge density. Since Equation \eqref{eq:1} implies $\mathbf{E}=-\nabla\phi$, Equation \eqref{eq:2} yields the Poisson equation
\beq\label{eq:3}
-\Delta\phi= \rho.
\eeq
Therefore, 
if $\rho=\d$, the solution of \eqref{eq:3} is 
$\phi(x)=1/(4\pi|x|)$,
but its energy is 
\[
\mathcal{H}=\frac{1}{2}\int_{\RT}|{\bf E}|^2\ dx=\frac{1}{2}\int_{\RT}|\n \phi|^2 \ dx=+\infty.
\]
When $\rho\in L^1(\RT)$, which is another relevant physical case, we cannot say, in general, that \eqref{eq:3} admits a solution with finite energy (see e.g. \cite{FOP} for a counterexample). In fact, it is easily seen from the Gagliardo-Nirenberg-Sobolev inequality, see e.g. \cite{LL}, that the mathematical assumption which implies the finiteness of the energy is $\rho\in L^{6/5}(\R^{3})$. This hypothesis does not cover all relevant physical cases. 

Maxwell's equations are variational by which we mean that they can be derived as the Euler equations of a Lagrangian. 
To avoid the violation of the {\em principle of finiteness}, Max Born proposed a nonlinear theory \cite{Bnat,B} starting from a modification of Maxwell's Lagrangian density. This theory is built on in analogy with Einstein's mechanics of special relativity. Indeed, one passes from Newton's mechanics to Einstein's mechanics by replacing the action function $\mathcal{L}_{\rm N} = \frac{1}{2}mv^2$ with $\mathcal{L}_{\rm E} =mc^2(1-\sqrt{1-v^2/c^2})$ as this last expression is one of the simplest which is real only when $v< c$ and gives the classical formulation in the limit of small velocities. 
By analogy, starting from Maxwell's Lagrangian density in the vacuum
\begin{equation}
\label{eq:Maxcc}
\mathcal{L}_{\rm M} 
= -\frac{F_{\mu\nu}F^{\mu\nu}}{4},
\end{equation}
where $F_{\mu\nu}=\partial_\mu A_\nu -\partial_\nu A_\mu $,  
$(A_0, A_1, A_2, A_3)= (\phi, -{\bf A})$ is the electromagnetic potential, $(x_0, x_1, x_2, x_3)= (t,x)$ and $\partial_j$ denotes the partial derivative with respect to $x_j$,
Born introduced the new Lagrangian density
\begin{equation}
\label{eq:Bcc}
\mathcal{L}_{\rm B}
= b^2\left(1-\sqrt{1+\frac{F_{\mu\nu}F^{\mu\nu}}{2b^2}}\right)
\sqrt{-\det(g_{\mu\nu})},
\end{equation}
where $b$ is a constant having the dimensions of $e/r_0^2$, $e$ and $r_{0}$ being respectively the charge and the {\em effective radius} of the electron. In this last formula, $g_{\mu\nu}$ is the Minkowski metric tensor with signature $(+---)$. Since Born's action, as well as Maxwell's action, is invariant only for the Lorentz group of transformations (orthogonal transformations), some months later, Born and Infeld introduced a modified version of the Lagrangian density
\begin{equation}
\label{eq:BIcc}
\mathcal{L}_{\rm BI}
= b^2\left(\sqrt{-\det(g_{\mu\nu})} - \sqrt{-\det\left(g_{\mu\nu} + \frac{F_{\mu\nu}}{b}\right)} \right),
\end{equation}
whose integral is now invariant for general transformations \cite{BInat,BI}.
Since the electromagnetic field $(\mathbf{E},\mathbf{B})$ is given by 
\[
\mathbf{B}=\nabla\times\mathbf{A}
\quad
\hbox{and}
\quad
\mathbf{E}=-\nabla\phi - \partial_t \mathbf{A},
\]
\eqref{eq:Maxcc}, \eqref{eq:Bcc} and \eqref{eq:BIcc} can be written respectively as
\[
\mathcal{L}_{\rm M}
= \frac{|\mathbf{E}|^2-|\mathbf{B}|^2}{2},
\quad
\mathcal{L}_{\rm B} 
= b^2\left(1-\sqrt{1-\frac{|\mathbf{E}|^2-|\mathbf{B}|^2}{b^2}}\right)
\]
and
\[
\mathcal{L}_{\rm BI} = b^2 \left(1-\sqrt{1-\frac{|\mathbf{E}|^2-|\mathbf{B}|^2}{b^2}-\frac{(\mathbf{E}\cdot\mathbf{B})^2}{b^4}}\right).
\]
In the electrostatic case, in which we are interested in this paper, we infer that
\[
\mathcal{L}_{\rm B} =  \mathcal{L}_{\rm BI} 
= b^2\left(1-\sqrt{1-\frac{|\mathbf{E}|^2}{b^2}}\right).
\]
As emphasized above, we recover Newton's classical mechanics from Einstein's special relativity for small velocities or when $c\to+\infty$. The same holds true with Born-Infeld formulation of electromagnetism: if $b\to+\infty$ or for electromagnetic fields having small intensities, both $\mathcal{L}_{\rm B}$ and  $\mathcal{L}_{\rm BI}$ reduce to Maxwell's Lagrangian density $\mathcal{L}_{\rm M}$.

In presence of a charge density $\rho$, we formally get the equation
\beq\label{eq:as}
-\operatorname{div}\left(\frac{\nabla \phi}{\sqrt{1-|\nabla \phi|^2/b^2}}\right)= \rho,
\eeq
which replaces the Poisson equation \eqref{eq:3}. 
This equation can also be obtained observing that the Born-Infeld theory distinguishes between the electric field $\mathbf{E}$ and the electric induction $\mathbf{D}$: the field $\mathbf{D}$ satisfies
\begin{equation}
\label{eq:divD}
\operatorname{div} \mathbf{D}= \rho
\end{equation}
and the fields $\mathbf{E}$ and $\mathbf{D}$ are related by
\begin{equation}
\label{eq:DE}
\mathbf{D}=\frac{\mathbf{E}}{\sqrt{1-(|\mathbf{E}|/b)^2}}.
\end{equation}
Substituting \eqref{eq:DE} in \eqref{eq:divD}, we recover \eqref{eq:as}.
Finite energy point particle solutions with $\delta$-function sources have been called {\em BIons} (see for example \cite{Gibb98}). When $\rho=\delta$, one can easily explicitly compute the solution,  see for example \cite{P}. 

\medbreak

From now on, for simplicity and without loss of generality, we fix $b=1$.
It is worth mentioning that the operator $Q^{-}$, defined as
\begin{equation}\label{Q-}
Q^{-}(\phi)=-\operatorname{div}\left(\frac{\nabla \phi}{\sqrt{1-|\nabla \phi|^2}}\right),
\end{equation}
also naturally appears in string theory, in particular in the study of $D$-branes (see e.g. \cite{Gibb98}) and, in classical relativity, where $Q^{-}$ represents the mean curvature operator in Lorentz-Minkowski space, see for instance \cite{BS,CY}.  In this last context, the following definition is standard.

\begin{definition}\label{def:spacelike}
	Let $\phi\in C^{0,1}(\Omega)$, with $\Omega\subset\RN$. We say that $\phi$ is
	\begin{itemize}
		\item {\em weakly spacelike} if $|\n \phi|\le 1$ a.e. in $\Omega$;
		\item {\em spacelike} $|\phi(x)-\phi(y)|<|x-y|$ whenever $x,y\in\Omega$, $x\neq y$ and the line segment $\overline{xy}\subset\Omega$;
		\item {\em strictly spacelike} if $\phi$ is spacelike, $\phi\in C^1(\Omega)$ and $|\n \phi|< 1$ in $\Omega$.
	\end{itemize}
\end{definition}

\medbreak
Our motivation in this paper is to study rigorously 
the boundary value problem
\begin{equation}\label{eq:BI}
\tag{$\mathcal{BI}$}
\left\{
\begin{array}{ll}
-\operatorname{div}\left(\displaystyle\frac{\nabla \phi}{\sqrt{1-|\nabla \phi|^2}}\right)= \rho \quad \hbox{in }\mathbb{R}^N,
\\[8mm]
\displaystyle\lim_{|x|\to \infty}\phi(x)= 0, 
\end{array}
\right.
\end{equation}
for general non-trivial charge distributions.
Assuming $N\ge 3$, we work in the functional space 
\begin{equation}\label{eq:spaceX}
\X=D^{1,2}(\RN)\cap \{\phi\in C^{0,1}(\RN) \mid \|\nabla \phi\|_\infty \le 1\},
\end{equation}
equipped with the norm 
defined by
\begin{equation*} \|\phi\|_{\X}:=\left(\int_{\RN}|\nabla\phi|^2 \ dx\right)^{1/2}.
\end{equation*}
More properties of this space are given in Section \ref{se:fs}. We recall that $D^{1,2}(\RN)$ is the completion of $C_c^\infty (\RN)$ with respect to above norm and we anticipate that $\X^*$, the dual space of $\X$, contains Radon measures as for instance superpositions of point charges or $L^1(\RN)$ densities.\\

For a $\rho\in \X^*$, {\em weak solutions} are understood in the following sense.
\begin{definition}\label{def:ws}
A {\em weak solution} of \eqref{eq:BI} is a function $\phi_\rho\in \X$ 
such that for all $\psi \in\X$, we have
\begin{equation}\label{eq:weakBI}
\irn \frac{\n \phi_\rho \cdot \n \psi}{\sqrt{1-|\nabla \phi_\rho|^2}}\, dx
=\langle \rho, \psi\rangle, 
\end{equation}
where $\langle\ , \ \rangle$ denotes the duality pairing between $\X^*$ and $\X$.
\end{definition}
Observe that the boundary condition at infinity is encoded in the functional space. 
We also emphasize that if $\rho$ is a distribution, the weak formulation of \eqref{eq:weakBI} extends to any test function $\psi \in C^{\infty}_{c}(\RN)$.

As Born-Infeld equation is formally the Euler equation of the action functional $I:\X\to \R$ defined by
\begin{equation}\label{eq:functionalBI}
I(\phi)=\irn \Big(1 - \sqrt{1-|\nabla \phi|^2}\Big) dx
- \langle \rho, \phi\rangle,
\end{equation}
we expect that one can derive existence and uniqueness of the solution from a variational principle. Furthermore, since $I$ is  bounded from below in $\X$ and strictly convex, one can look for the solution as the minimizer of $I$ in $\X$ by the direct methods of the Calculus of Variations. However, one needs to pay attention to the lack of regularity of the functional when $ \|\nabla \phi\|_\infty = 1$. Hence, as in convex optimization or in the Calculus of Variations for non smooth functionals (see for example \cite{EkTe,S}), it is natural to relax the notion of critical point: more precisely we say that $\phi_\rho\in \X$ is 
a {\em critical point in  weak sense} for the functional $I$ if $0$ belongs to the subdifferential of $I$ at $\phi_\rho$ (see Definition \ref{def}), which, in our case, simply amounts to ask that $\phi_\rho$ is a minimum for the functional $I$, see Remark \ref{rema}. 
We refer to Subsection \ref{se:fs} for more details. 

\medbreak

We first prove (see Subsection \ref{se:fs}) existence and uniqueness for the relaxed problem.
\begin{theorem}\label{thm:relaxedBI}
For any  $\rho\in \X^{*}$, there exists a unique $\phi_\rho$ which minimizes $I$. This is the unique critical point in weak sense of the functional $I$.
\end{theorem}

Up to our knowledge, it is not known in the literature whether the weak formulation \eqref{eq:weakBI}  holds or not for critical point in weak sense under the mere assumption $\rho\in \X^{*}$. This question has motivated several publications in the past years. In \cite{FOP}, the authors deal with the second order expansion of the non smooth part of the functional, assuming $\rho\in L^{1}(\RN)$, see Subsection \ref{sse:fop}. In \cite{K}, the author considers the special case $\rho = 4\pi\sum_{i=1}^k \a_i \d_{x_i}$, however there is a gap in the proof, see Section \ref{se:delte} for more details.
 
\medbreak

The first case in which we can deduce the weak formulation is when $\rho\in \X^{*}$ is a radially distributed charge, see Section \ref{Sec:radial} for the precise statement. 

\begin{theorem}
\label{thm:exsol}
If $\rho\in  \X^*$ is radially distributed, then there exists a unique (radial) weak solution $\phi_\rho\in \X$ of \eqref{eq:BI}.
\end{theorem} 

Under stronger assumptions on $\rho$, still assuming radial symmetry of the source, we investigate the regularity of the solution and we partially recover the regularity of Poisson equation, for which we refer to \cite{LL}, (see Theorem \ref{thm:reg}).

\medbreak

In the same Section \ref{Sec:radial}, without symmetry assumptions, we consider the case of locally bounded source. 
\begin{theorem}\label{th:limitato}
If $\rho\in L^\infty_{\rm loc}(\mathbb{R}^N)\cap\X^*$, then $\phi_\rho$ is a (locally strictly) space-like weak solution of \eqref{eq:BI}.
\end{theorem}

The case of a superposition of charges, namely 
$$\rho=\sum_{i=1}^k a_i \d_{x_i},$$  
where $a_i\in \R$ and $x_i\in \RN$, for $i=1,\ldots,k$, $k\in\N_{0}$, 
is studied in Section \ref{se:delte}. We first identify, as in \cite{K}, the possible singular points of the solutions. Basically,  the minimum $\phi_\rho$ of $I$
is always strictly spacelike on $\RN\setminus\G$, where 
$$ \G=\bigcup_{ i\neq j}\overline{x_i x_j}.$$ 
We then prove that the minimizer is a distributional solution away from the charges and if the intensities are small or if the charges are sufficiently far away from each other, then $\phi_{\rho}$ is strictly spacelike on $\RN\setminus\{x_1,\ldots,x_k\}$ and singular only and exactly at the location of the charges, i.e.
\[
\lim_{x\to x_i}|\n \phi_\rho(x)|=1. 
\]
\begin{theorem}\label{th:deltafinale}
Assume $\rho=\sum_{i=1}^k a_i \d_{x_i}$, where $a_i\in \R$ and $x_i\in \RN$, for all $i=1,\ldots,k$. 
Then $\phi_\rho$ is a distributional solution of the Euler-Lagrange equation in $\RN\setminus\{x_1,\ldots,x_k\}$.
Namely, for every $\psi \in C_{c}^\infty(\RN\setminus\{x_1,\ldots,x_k\})$, we have
\begin{equation*} 
\irn \frac{\n \phi_\rho \cdot \n \psi}{\sqrt{1-|\nabla \phi_\rho|^2}}\, dx
=0.
\end{equation*}
It is a classical solution of the equation in $\RN\setminus\G$, namely $\phi_\rho\in C^\infty(\RN\setminus\Gamma)$ and 
$$-\dv \left( \dfrac{\n \phi}{\sqrt{1-|\n \phi|^2}}\right)=0$$
in the classical sense in $\RN\setminus\G$.
Moreover, 
\begin{enumerate}
\item \label{it:t161}for any fixed $x_i\in \RN$, $i=1,\ldots,k$, there exists $\s=\s(x_1,\ldots,x_k)>0$ such that if 
$$\max_{ i=1,\ldots,k}|a_i|<\s,$$ then $\phi_\rho$ is a classical solution in $\RN\setminus\{x_1,\ldots,x_k\}$; 
\item \label{it:t162}for any $a_i\in \R$, $i=1,\ldots,k$, there exists $\tau=\tau(a_1,\ldots,a_k)>0$ such that if
$$\min_{ i,j=1,\ldots,k, \ i\neq j}|x_i-x_j|>\tau,$$ 
then $\phi_\rho$ is a classical solution in $\RN\setminus\{x_1,\ldots,x_k\}$. 
\end{enumerate}
In these last cases, 
$\phi_\rho\in C^\infty (\RN\setminus\{x_1,\ldots,x_k\})$, it is strictly spacelike on $\RN\setminus\{x_1,\ldots,x_k\}$ and 
\[
\lim_{x\to x_i}|\n \phi_\rho(x)|=1. 
\]
\end{theorem}

Moreover, in some cases, even if we do not know that minimum $\phi_\rho$ of the functional $I$ is actually a weak solution of \eqref{eq:BI}, we can say that it
is the {\em limit} of solutions of approximated problems obtained by modifying the differential operator or mollifying the charge density. This is the concern of Section \ref{se:approx}.

We conclude in Section \ref{sec:conclusion} by additional results and comments. In particular, we mention how our method completes some previous studies \cite{DP,M,Wang,yu} of a field, governed by the nonlinear Klein-Gordon equation, coupled with the electromagnetic field whose Lagrangian density is given by \eqref{eq:Bcc} or \eqref{eq:BIcc}, by means of the Weil covariant derivatives.

\medbreak

We finally mention that the operator $Q^{-}$ has been studied in other situations by many authors in the recent years. We refer to \cite{A,BDD} for some results in $\R^{N}$ and to \cite{Maw} which provides further references for boundary value problems in a bounded domain. In particular Bartnik and Simon have been among the first to deal with this type of differential operator and some of the ideas from \cite{BS} are fundamental in our arguments. Observe that the results of Bartnik and Simon have been used to deal with the Dirichlet problem in bounded or unbounded 
domains, see \cite{Kl1,Kl2,KM}, assuming the existence of prescribed singularities inside the domain. The results therein can be used for a superposition of charges, if the intensities of the charges are sufficiently small, but they are restricted to domains with boundary and with a given Dirichlet condition.

\medbreak

We close the introduction with some notations: $C$ denotes a generic positive constants which can change from line to line, $B_R$ is the ball centered in $0$ with radius $R>0$ and $\o_N$ denotes the measure of the $(N-1)$-dimensional unitarian sphere. For every $1\le p<N$, $p^*$ is the critical Sobolev exponent in the Sobolev inequality, namely $p^*= Np/(N-p)$.

\subsection*{Acknowledgement}
This work has been partially carried out during a stay of P.D. and A.P. in Bruxelles. They would like to
express their deep gratitude to the D\'epar\-te\-ment de Math\'ematique, Universit\'e libre de Bruxelles, for the support and warm hospitality. D.B. acknowledge the support of INDAM for his visits at the Politecnico di Bari where parts of this work have been achieved.

\section{A relaxed formulation via nonsmooth analysis }\label{se:relaxed}

This section is devoted to the proof of Theorem \ref{thm:relaxedBI} and to useful properties of the minimum $\phi_\rho$. We start with the functional setting and we recall some well-known facts from convex analysis and non smooth critical point theory.

\subsection{Functional setting and the existence of the critical point in weak sense}\label{se:fs}

We start with some properties of the ambient space $\X$ defined in \eqref{eq:spaceX}. The proof follows from standard arguments that we give for completeness.

\begin{lemma}
\label{lemma21}
The following assertions hold:
\begin{enumerate}[label=(\roman*),ref=\roman*]
\item \label{it:w1p}$\X$ is continuously embedded in $W^{1,p}(\RN)$, for all $p\ge 2^*=2N/(N-2)$;
\item \label{it:embLinf}  $\X$ is continuously embedded in $L^\infty(\RN)$;
\item \label{it:C0} if $\phi\in \X$, then $\lim_{|x|\to \infty} \phi(x)=0$;
\item \label{it:wc} $\X$ is weakly closed;
\item \label{it:compact}  if $(\phi_n)_n\subset\X$ is bounded, there exists $\bar \phi\in \X$ such that, up to a subsequence, $\phi_{n}\rightharpoonup \bar \phi$ weakly in $\X$ and uniformly on compact sets.
\end{enumerate}
\end{lemma}

\begin{proof}
By definition, if $\phi\in\X$, then for every $q\ge 2$, we have $|\n \phi|\in L^q(\RN)$ whence  
$\phi\in L^{q^*}(\RN)$ for $q\in [2,N)$ by Sobolev inequality. Hence $\phi\in W^{1,p}(\RN)$ for $p\ge 2^*$. The continuity of the imbedding is clear and then the  proof of assertion $(\ref{it:w1p})$ is complete.
\\
Assertions $(\ref{it:embLinf})$ and $(\ref{it:C0})$ are direct consequences of $(\ref{it:w1p})$, Morrey-Sobolev inequality.  
\\
As regards $(\ref{it:wc})$, since $\X$ is convex,  it is sufficient to show that $\X$ is closed with respect to the strong topology. Take $(\phi_n)_{n}\subset \X$ such that $\phi_n \to \phi$ in $\X$. Then we have 
$$|\phi_{n}(x)-\phi_{n}(y)|\le |x-y|,$$
for all $x,y\in \RN$. Since $\phi_n \to \phi$ uniformly in $\RN$ by $(\ref{it:embLinf})$, we conclude that $\|\n\phi\|_\infty\le 1$.
We finally prove $(\ref{it:compact})$. Since $(\phi_{n})_{n}$ is a bounded sequence in $D^{1,2}(\RN)$, it contains a weakly converging subsequence that we still denote by $(\phi_{n})_{n}$. By $(\ref{it:wc})$, the weak limit $\bar\phi$ belongs to $\X$ and by Ascoli-Arzel\`a Theorem, the convergence is uniform on compact sets. 
\end{proof}

We now give some properties of the functional $I$. The simple inequality
\beq\label{ineq}
\frac 12 t\le 1-\sqrt{1-t}\le t, \quad \hbox{ for all }t\in [0,1]
\eeq
will be useful. 

\begin{lemma}\label{blabla}
The functional $I: \X\to \mathbb{R}$ is 
\begin{enumerate}[label=(\roman*), ref=\roman*]
\item \label{it:bb} bounded from below,
\item \label{it:coer} coercive,
\item \label{it:cont} continuous,
\item \label{it:conv} strictly convex,
\item \label{it:wlsc} weakly lower semi-continuous.
\end{enumerate} 
\end{lemma}
\begin{proof}
Using \eqref{ineq}, we infer that 
\[
I(\phi)
\ge \frac{1}{2} \|\n \phi\|_2^2 -\|\rho\|_{\X^*} \|\n \phi\|_2
\]
for every $\phi\in\X$ and this yields $(\ref{it:bb})$ and $(\ref{it:coer})$.
As regards $(\ref{it:cont})$, we only need to prove  that $J:\X \to \R$, defined by
\[
J(\phi)=\irn  \Big(1 - \sqrt{1-|\nabla \phi|^2}\Big) dx,
\]
is continuous. To this end, we consider a sequence $(\phi_n)_n\subset\X$ that converges to $\phi$ in $\X$. Then, up to a subsequence, $\n \phi_n \to \n \phi$ a.e. in $\RN$ and there exists $w\in L^1(\RN)$ such that $|\n \phi_n|^2\le w$,  $|\n \phi|^2\le w$ a.e. in $\RN$. Thus, by \eqref{ineq},
\[
\left|(1-\sqrt{1-|\n \phi_n|^2}) - (1-\sqrt{1-|\n \phi|^2})\right| 
\le
|\n \phi_n|^2 + |\n \phi|^2 \le 2 w
\]
and then, by Lebesgue's Dominated Convergence Theorem we have that $J(\phi_n)\to J(\phi)$. It is straightforward to check that the convergence holds for the whole sequence.
\\
The strict convexity of $J$ follows from the strict convexity of the real function $y\in B_1\mapsto 1-\sqrt{1-|y|^2}$. Since $J$ is continuous 
and convex whereas $\rho$ is continuous with respect to the weak convergence, the weak lower semi-continuity holds.
\end{proof}

As a consequence of the previous properties, we get the existence of a unique minimizer.
\begin{proposition}\label{pr:min}
The infimum $m=\inf_{\phi\in \X}I(\phi)$ is achieved by a unique  $\phi_\rho \in \X\setminus\{0\}$.
\end{proposition}

\begin{proof}
The existence and uniqueness follow from Lemma \ref{blabla}. Therefore, we only have to show that $\phi_\rho$ is nontrivial or equivalently that $m<0$. Taking $\phi\in \X$ such that $\langle \rho, \phi \rangle>0$, we compute
\[
I(t \phi)\le t^2 \|\n \phi\|_2^2 - t\langle \rho, \phi \rangle<0, 
\]
for $t>0$ small enough, whence $m<0$.
\end{proof}

We now recall some classical definitions from convex analysis, see \cite{S}. 
\begin{definition}\label{def}
Let $X$ be a real Banach space and $\Psi:X\to (-\infty,+\infty]$ be a convex lower semicontinuous function. Let $D(\Psi)=\{u\in X\mid \Psi(u)<+\infty\}$ be the effective domain of $\Psi$. For $u \in D(\Psi)$, the set
\[
\de \Psi(u)=\{u^*\in X^* \mid \Psi(v)-\Psi(u)\ge \langle u^*,v-u\rangle, \ \forall v\in X\}
\]
is called the {\em subdifferential} of $\Psi$ at $u$. If, moreover, we consider a functional $I=\Psi+\Phi$, with $\Psi$ as above and $\Phi\in C^1(X,\R)$, then $u\in D(\Psi)$ is said to be {\em critical in weak sense} if $-\Phi'(u)\in \de \Psi(u)$, that is
\[
\langle\Phi'( u),v-u\rangle+ \Psi(v)-\Psi(u)\ge 0, \ \forall v\in X.
\]
\end{definition}

\begin{remark}\label{rema}
Observe that, according to the previous definition, $\phi_\rho$ is a critical point in weak sense for the functional $I$ if and only if, for any $\phi \in \X$ we get
\[
\irn \Big(1 - \displaystyle\sqrt{1-|\nabla \phi|^2}\Big) dx -\irn \Big(1 - \sqrt{1-|\nabla \phi_\rho|^2}\Big) dx 
\ge  \langle \rho,\phi - \phi_\rho \rangle,
\]
which is simply equivalent to requiring that $\phi_\rho$ is a minimum for $I$.
\end{remark}

Proposition \ref{pr:min} and Remark \ref{rema} leads easily to Theorem \ref{thm:relaxedBI}.

\subsection{Further properties}

Here we present some properties that are useful in the sequel. We start by recalling the convexity of the functional implies that weak solutions are minimizers and therefore we deduce the uniqueness of the weak solution (if any). As before, for $\rho\in\X^{*}$, we denote the unique minimizer of $I$ in $\X$ by $\phi_\rho$.

\begin{proposition}\label{pr:unicita}
Assume $\rho\in\X^{*}$. If $\phi\in \X$ is a weak solution of \eqref{eq:BI}, then $\phi = \phi_\rho$.
\end{proposition}

\begin{proof}
If $\phi\in \X$ satisfies \eqref{eq:weakBI}, it is easy to see that $|\nabla\phi(x)|<1$ for a.e. $x\in\mathbb{R}^N$. 
By convexity, we get
\begin{equation}
	\label{eq:puntconv}
	1-\sqrt{1-|\nabla\phi_\rho(x)|^2} 
	\ge 
	1- \sqrt{1-|\nabla\phi(x)|^2}
	+\frac{\nabla\phi(x)\cdot(\nabla\phi_\rho(x)-\nabla\phi(x))}{ \sqrt{1-|\nabla\phi(x)|^2}}
\end{equation}
	for a.e. $x\in\mathbb{R}^N$. 
	Moreover, again since $\phi$ satisfies \eqref{eq:weakBI}, we have
	\begin{equation}
	\label{eq:ineqsolperphistar}
	\irn \frac{|\nabla\phi|^2}{\sqrt{1-|\nabla\phi|^2}}\ dx
	- \irn \frac{\nabla\phi\cdot \nabla\phi_\rho}{\sqrt{1-|\nabla\phi|^2}}\ dx
	=
	\langle \rho, \phi - \phi_\rho \rangle.
	\end{equation}
Combining \eqref{eq:puntconv} and \eqref{eq:ineqsolperphistar}, we conclude that $I(\phi_\rho) \ge I(\phi)$.
Uniqueness of the minimizer $\phi_\rho$ of $I$ in $\X$ leads to the conclusion.
\end{proof}

Proposition \ref{pr:unicita} relies on the fact that weak solutions are minimizers. 
Actually, the fact that the weak solutions of \eqref{eq:BI} with Dirichlet boundary conditions on bounded domains with regular sources are unique minimizers of $I$, was already known in \cite{BS}.

A question now arises naturally: does the reverse statement hold? Namely, is it true that the unique minimizer $\phi_\rho$ is always a weak solution of \eqref{eq:BI}? We are not able to answer this question in its full generality but we conjecture a positive answer and the following statement 
goes in that direction.

\begin{proposition}\label{pr:sol-deb}
Assume $\rho\in \X^{*}$ and let $\phi_\rho$ be the unique minimizer of $I$ in $\X$. Then
 $$E=\{x\in \RN \mid |\n \phi_\rho|=1\}$$ is a null set (with respect to Lebesgue measure) and the function $\phi_\rho$ satisfies 
\beq\label{plipli}
\irn \frac{  |\n  \phi_\rho|^2}{ \sqrt{1- |\n \phi_\rho|^2}}\, dx
\le \langle \rho,\phi_\rho \rangle.
\eeq
Moreover, for all $\psi\in \X $, we have the variational inequality
\begin{equation}
\label{eq:ineqcc}
\irn \frac{  |\n  \phi_\rho|^2}{ \sqrt{1- |\n \phi_\rho|^2}}\, dx
-\irn \frac{ \n \phi_\rho \cdot \n \psi}{ \sqrt{1- |\n \phi_\rho|^2}}\, dx
\le \langle \rho,\phi_\rho - \psi \rangle.
\end{equation}
\end{proposition}

\begin{proof}
Since for every $t\in [0,1]$ and $\psi\in \X$, $\phi_t=\phi_\rho +t (\psi - \phi_\rho)\in \X$, we have $I(\phi_\rho)\le I(\phi_t)$, namely
\begin{equation}
\label{eq:juut}
\irn \Big( 1-\sqrt{1-|\nabla\phi_\rho(x)|^2} \Big) dx -  \irn \Big( 1-\displaystyle\sqrt{1-|\nabla\phi_{t}(x)|^2} \Big) dx
\le t \langle \rho,\phi_\rho-\psi \rangle.
 \end{equation}
In the particular case $\psi=0$, we have 
\[
\int_{E^c} \frac{\sqrt{1-(1-t)^2|\n \phi_\rho|^2}-\sqrt{1-|\n \phi_\rho|^2}}{t}\, dx
+\frac{\sqrt{2-t}}{\sqrt{t}}|E|
\le \langle \rho,\phi_\rho \rangle
\]
for all $t\in(0,1]$. Since both terms in the left hand side are nonnegative, we infer that $|E|=0$
and moreover
\[
\irn \frac{  (2-t)|\n  \phi_\rho|^2}{ \sqrt{1- |\n \phi_\rho|^2}+ \sqrt{1- (1-t)^2|\n \phi_\rho|^2}}\, dx
\le \langle \rho,\phi_\rho \rangle.
\]
By Fatou's Lemma, letting $t$ go to zero, we get \eqref{plipli} and therefore
\beq\label{eq:L1}
\frac{  |\n  \phi_\rho|^2}{ \sqrt{1- |\n \phi_\rho|^2}}\in L^1(\RN). 
\eeq
Considering now $\psi\neq 0$ in \eqref{eq:juut}, as $|E|=0$, we deduce that
\begin{equation}
\label{eq:uut}
\irn \frac{ (2-t) |\n \phi_\rho|^2  - 2(1-t)\n \phi_\rho \cdot \n \psi - t  |\n \psi |^2}{ \sqrt{1- |\n \phi_\rho|^2}+ \sqrt{1- |\n \phi_t|^2}}\, dx\le
\langle \rho,\phi_\rho- \psi \rangle.
\end{equation}
Aiming to apply Lebesgue's Dominated Convergence Theorem in \eqref{eq:uut}, we first notice that
\[
 \left|\frac{ (2-t) |\n \phi_\rho|^2  - 2(1-t)\n \phi_\rho \cdot \n \psi - t  |\n \psi |^2}{ \sqrt{1- |\n \phi_\rho|^2}+ \sqrt{1- |\n \phi_t|^2}}\right| 
 \le C\left(\frac{ |\n \phi_\rho |^2}{ \sqrt{1- |\n \phi_\rho|^2}}
+ \frac{ |\n \psi|^2}{ \sqrt{1- |\n \phi_\rho|^2}}\right) 
\]
The first term of the right hand side is $L^1(\RN)$ by \eqref{eq:L1}. To estimate the second one, we observe that 
$$
\int_{A_{\rho}} \frac{ |\n \psi|^2 }{ \sqrt{1- |\n \phi_\rho|^2}}\, dx
\le  C\int_{\RN} |\n \psi|^2\, dx
$$
and 
$$\int_{A_{\rho}^{c}} \frac{ |\n \psi|^2}{ \sqrt{1- |\n \phi_\rho|^2}}\, dx \le C \int_{\RN} \frac{ |\n \phi_\rho |^2}{ \sqrt{1- |\n \phi_\rho|^2}}\, dx,$$
where we have set $A_{\rho}:=\{x\in \R^{N}\mid |\n \phi_\rho|\le 1/2\}$.
Since we are now allowed to take the limit as $t\to 0^+$ in \eqref{eq:uut}, we get \eqref{eq:ineqcc}.
\end{proof}

\begin{remark}
\label{rem:nehary}
If $\phi_\rho$ satisfies further
\begin{equation*}
\irn \frac{  |\n  \phi_\rho|^2}{ \sqrt{1- |\n \phi_\rho|^2}}\, dx
= \langle \rho,\phi_\rho \rangle,
\end{equation*}
then, by \eqref{eq:ineqcc}, it is easy to see that $\phi_\rho$ is a weak solution of \eqref{eq:BI}.
\end{remark}

\begin{remark}
It is clear from the proof of Proposition \ref{pr:sol-deb} that it is enough to assume $\psi\in \D $ together with $|\nabla\psi|\in L^\infty(\RN)$ to get
	\[
	\irn \frac{\nabla \phi_\rho \cdot\nabla\psi}{\sqrt{1-|\n\phi_\rho|^2}}\ dx\in \R.
	\]
\end{remark}

The next lemma states a useful convergence to prove Theorem \ref{thm:exsol}.

\begin{lemma}\label{le:cont}
Assume $\rho\in \X^{*}$ and let $\phi_\rho$ be the unique minimizer of $I$ in $\X$. If $(\psi_n)_n\subset \D $ is such that $\|\n \psi_n\|_\infty\le C$ for some $C>0$ and $\psi_{n}\to\psi$ in $\D$ then, up to a subsequence, 
\[
\lim_{n\to\infty}\irn \frac{ \n \phi_\rho \cdot \n \psi_n}{ \sqrt{1- |\n \phi_\rho|^2}} \, dx
= \irn \frac{ \n \phi_\rho \cdot \n \psi}{ \sqrt{1- |\n \phi_\rho|^2}} \, dx.
\]
\end{lemma}

\begin{proof}
Keeping the notation $A_{\rho}=\{x\in \R^{N}\mid |\n \phi_\rho|\le 1/2\}$, we have
$$
\int_{A_{\rho}} \frac{ \n \phi_\rho \cdot (\n \psi_n - \n \psi) }{ \sqrt{1- |\n \phi_\rho|^2}} \, dx
\le C \left(\irn |\n \phi_\rho |^{2} \, dx\right)^{1/2} \left(\irn |\n (\psi_n - \psi)|^2 \, dx\right)^{1/2}$$
whereas, on $A_{\rho}^{c}$,
$$
\left|\frac{ \n \phi_\rho \cdot (\n \psi_n - \n \psi) }{ \sqrt{1- |\n \phi_\rho|^2}} \right|\le 
 C \frac{ |\n \phi_\rho |^{2}}{ \sqrt{1- |\n \phi_\rho|^2}}.
$$
Recalling that 
$$\frac{ |\n \phi_\rho |^2 }{ \sqrt{1- |\n \phi_\rho|^2}} \in L^1(\RN),$$
we can apply Lebesgue's Dominated Convergence Theorem on $A_{\rho}^{c}$ so that the conclusion follows.
\end{proof}

In the spirit of \cite[Lemma 1.2]{BS}, we next provide a monotonicity property of the map $\rho\in \X^* \mapsto \phi_\rho$, where $\phi_\rho$ is the unique minimum associated to $\rho$. We first order the elements of $\X^{*}$.

\begin{definition}
\label{def:rhopos}
Let $\rho_1,\rho_2\in \X^*$. We say that  $\rho_1\le \rho_2$, if we have $\langle \rho_1,\vfi\rangle\le \langle \rho_2,\vfi\rangle$ for any $\varphi\in \X$ with $\vfi\ge 0$.
\end{definition}

The next lemma is a {\em comparison principle} for minimizers.

\begin{lemma}\label{monotonia}
If $\rho_1,\rho_2\in \X^*$ are such that $\rho_1\le \rho_2$, then $\phi_{\rho_{1}}(x)\le \phi_{\rho_{2}} (x)$ for all $x\in \RN$.
\end{lemma}

\begin{proof}
Let $\phi_i= \phi_{\rho_{i}}$. Suppose by contradiction that the open set $\O^+:=\{x\in \RN \mid \phi_1(x)>\phi_2(x)\}$ is non-empty. Of course, by continuity, we have $\phi_1=\phi_2$ on $\de \O^+$. We define $\O^-:=\RN \setminus \O^+$, which may be the empty set and we introduce the functionals $I_{1},I_{2}:\X\to\R$, where
$$
I_i(\phi)=\irn \Big(1-\sqrt{1-|\nabla \phi|^2}\Big) \ dx- \langle\rho_i, \phi\rangle.
$$
We also fix the notation 
$$J_\pm(\phi)=\int_{\O^\pm} \Big(1-\sqrt{1-|\nabla \phi|^2}\Big)\ dx.$$
Let
\[
\psi_1(x)=\left\{
\begin{array}{ll}
\phi_2 (x)& \hbox{in }\O^+,
\\ 
\phi_1 (x)& \hbox{in }\O^-,
\end{array}\right.
\]
and
\[
\psi_2(x)=\left\{
\begin{array}{ll}
\phi_1(x) & \hbox{in }\O^+,
\\ 
\phi_2(x) & \hbox{in }\O^-.
\end{array}\right.
\]
Clearly, $\psi_i\in \X$ and $\psi_2 - \phi_2 = \phi_1 - \psi_1 \ge 0$. Then
\begin{align*}
I_2(\psi_2)
&=
J_+(\phi_1) + J_-(\phi_2) - \langle\rho_2,\psi_2\rangle\\
&=
I_2(\phi_2) 
+J_+(\phi_1) -J_+(\phi_2)   - \langle\rho_2,\phi_1 - \psi_1 \rangle\\
&  \le I_2(\phi_2) 
+J_+(\phi_1) -J_+(\phi_2)   - \langle\rho_1,\phi_1 - \psi_1 \rangle \\
& =
I_2(\phi_2) 
+I_1(\phi_1)  -I_1(\psi_1).
\end{align*}
Now, since $\phi_{1}$ is the unique minimizer of $I_{1}$ and $\psi_{1}$ does not coincide with $\phi_{1}$ on an open set, we conclude that 
$$I_2(\phi_2) 
+I_1(\phi_1)  -I_1(\psi_1)< I_2(\phi_2),$$
and we reach a contradiction with the minimality of $\phi_{2}$. 
\end{proof}

We conclude this section with the following property which will be useful in the sequel.

\begin{lemma}\label{compactset}
Let $(\rho_n)_n \subset \X^*$ be a bounded sequence. Then there exists $\bar \phi\in \X$ such that  $\phi_{\rho_n}\rightharpoonup \bar \phi$ weakly in $\X$ and uniformly on compact sets.
\end{lemma}

\begin{proof}
For short, we set  $\phi_n = \phi_{\rho_n}$ and
$$
I_n(\phi)=\irn \Big(1-\sqrt{1-|\nabla \phi|^2}\Big) \ dx- \langle\rho_n, \phi\rangle.
$$ 
By Assertion $(\ref{it:compact})$ of Lemma \ref{lemma21}, it is enough to show that $(\phi_{n})_{n}$ is bounded. 
Since $(\rho_n)_n$ is a bounded sequence in $\X^*$ and  $I_{n}(\phi_n)\le 0$, we have
\[
\frac{1}{2}\|\n \phi_n\|^2_2
\le \irn \Big( 1- \sqrt{1-|\n \phi_n|^2}\Big) \ dx
\le \langle \rho_n ,\phi_n \rangle
\le C\|\n \phi_n\|_2, 
\]
namely $(\phi_n)_n$ is a bounded sequence in $\X$. 
\end{proof}

\section{Weak formulation for radially distributed or bounded charge densities}\label{Sec:radial}

\subsection{The case of a radial charge}
We now turn our attention to radially distributed charges. In this case, we are able to recover the weak formulation from the relaxed equation. Indeed we build a dense set, within the set of test functions, of admissible variations. The argument is borrowed from \cite{ST} where it was used to handle a monotonicity constraint. 

\medbreak

We first precise the meaning of radially distributed charge density. For $\tau \in O(N)$,  $\phi\in \X$ and $\rho\in  \X^*$, we define $\phi^\tau\in \X$ as 
$\phi^\tau(x)=\phi(\tau x)$, for all $x\in \RN$, and $\rho^\tau\in  \X^*$ as $\langle \rho^\tau,\psi \rangle=\langle \rho,\psi^\tau \rangle$,
for all $\psi\in \X$.

\begin{definition}\label{def:rad}
We say that $\rho\in \X^*$ is radially distributed if $\rho^\tau=\rho$, for any $\tau \in O(N)$.
\end{definition} 
We next define
\[
\X_{\rm rad}=\{\phi\in \X \mid \phi^{\tau} = \phi \text{ for every }\tau \in O(N)\}
\]
and if $\phi\in \X_{\rm rad}$, in order to simplify the notations, we keep $\phi$ to denote the function $r=|x|\in\R_+\mapsto \phi(r)$ of a single real variable. We make furthermore a similar identification for the radially distributed maps $\rho \in \X^*$.

\medbreak 

Given these definitions, we are ready to prove that radial minimizers are weak solutions.

\begin{proof}[Proof of Theorem \ref{thm:exsol}]
As starting point, let us show that $\phi_\rho \in \X_{\rm rad}$. Indeed, for any $\tau \in O(N)$, since $\rho$ radially distributed, we have
$I(\phi_\rho^\tau)=I(\phi_\rho)$ and so we conclude by the uniqueness of the minimum.

\medbreak

\noindent We now prove that $\phi_{\rho}$ is a weak solution of \eqref{eq:BI}. Define, for $k\in\N\setminus\{0\}$, the sets
\[
E_k=\left\{r\ge 0\;\vline\;|\phi_\rho'(r)| \ge 1-\frac{1}{k}\right\}.
\]
Since 
$$E=\{r\ge 0 \mid |\phi_\rho'(r)|=1\}$$ 
is a null set, we have that $\left|\bigcap_{k\ge 1} E_k\right|=0$.
\\
Take $\psi\in \X_{\rm rad}\cap C^\infty_c(\RN)$ with $\supp \psi\subset [0,R]$ and let
\[
\psi_k(r)=-\int_{r}^{+\infty} \psi'(s) [1-\chi_{E_k}(s)] ds.
\]
Of course, by construction, $\supp \psi_k\subset [0,R]$, for any $k\ge 1$. Moreover, if $|t|$ is sufficiently small, then $\phi_\rho+t\psi_k\in \X$.
Indeed, since $(\phi_\rho+t\psi_k)'= \phi_\rho'+t\psi'[1-\chi_{E_k}]$,
then, if $r\in E_k$
\[
|(\phi_\rho+t\psi_k)'(r)|=|\phi_\rho'(r)|\le 1,
\]
otherwise, taking  $|t|\le\frac{1}{k\|\psi'\|_\infty}$, we have
\[
|(\phi_\rho+t\psi_k)'(r)|
\le
|\phi_\rho'(r)| + |t| \|\psi'\|_\infty
<
1 - \frac{1}{k} + |t| \|\psi'\|_\infty
\le 1.
\]
Now, since $\phi_\rho$ is the minimizer of $I$, arguing as in the proof of Proposition \ref{pr:sol-deb}, we infer that
\begin{equation}\label{eq:ELeq}
\lim_{t\to 0} \frac{I(\phi_\rho+t\psi_k)-I(\phi_\rho)}{t}
=\o_N\int_0^{+\infty} \frac{\phi_\rho'\psi'}{\sqrt{1-|\phi_\rho'|^2}}[1-\chi_{E_k}] r^{N-1} dr
- \langle \rho,\psi_k\rangle = 0,
\end{equation}
for every $k\ge 1$. Since $E_{k+1}\subset E_k$ and $|E_k|\to 0$, as $k\to +\infty$, we have that $\chi_{E_k}\to 0$ a.e. in $\mathbb{R}^N$ and so, by Lebesgue's Dominated Convergence Theorem, 
\[
\int_0^{+\infty} \frac{\phi_\rho'\psi'}{\sqrt{1-|\phi_\rho'|^2}}[1-\chi_{E_k}] r^{N-1} dr
\to
\int_0^{+\infty} \frac{\phi_\rho'\psi'}{\sqrt{1-|\phi_\rho'|^2}} r^{N-1} dr.
\]
Moreover, since it is easily seen that $\psi_k \to \psi$ in $\X$, we have
\[
\langle \rho,\psi_k\rangle \to \langle \rho,\psi\rangle,
\]
as $k\to \infty$.
Hence, for any $\psi\in \X_{\rm rad}\cap C^\infty_c(\RN)$, taking the limit in \eqref{eq:ELeq} as $k\to\infty$, we conclude that
\beq\label{solorad}
\irn \frac{\n \phi_\rho \cdot \n \psi}{\sqrt{1-|\n \phi_\rho|^2}} \ dx
= \langle \rho,\psi\rangle.
\eeq
We finally show that \eqref{solorad} holds also for any $\psi \in \X_{\rm rad}$. 
Given $\psi \in \X_{\rm rad}$, it is easy to see that there exists $(\psi_n)_n\subset C^\infty_c(\RN)$, $\psi_n$ radially symmetric such that $\psi_n\to \psi$ in $\D$ and with $\|\n \psi_n\|_\infty\le C$. Indeed it is sufficient to consider $\psi_n =\zeta_n \ast (\chi_n \psi)$, where  $\zeta_n$ are smooth radially symmetric mollifiers with compact support, and $\chi_n=\chi(\cdot/n)$, where $\chi:\RN \to \R$ is a smooth radially symmetric function such that
\[
\chi(x)=\left\{
\begin{array}{ll}
1 & \hbox{if }|x|\le 1,
\\
0 & \hbox{if }|x|\ge 2.
\end{array}
\right.
\]
Then, by \eqref{solorad}, we have
\[
\irn \frac{\n \phi_\rho \cdot \n \psi_n}{\sqrt{1-|\n \phi_\rho|^2}} dx
= \langle \rho,\psi_n\rangle,
\]
for any $n\ge 1$, and since $\psi_n\to \psi$ in $\D$, with $\|\n \psi_n\|_\infty\le C$, we infer from Lemma \ref{le:cont} that \eqref{solorad} holds for any $\psi \in \X_{\rm rad}$. To conclude, it remains to show that \eqref{solorad} holds for any $\psi \in \X$. Since $\phi_\rho$ is radially symmetric, we can take $\psi=\phi_\rho$ in \eqref{solorad} and therefore Remark \ref{rem:nehary} allows to conclude.
\end{proof}

Assuming further hypotheses on $\rho$, we can prove that the solution is $C^{1}$.
\begin{theorem}
\label{thm:reg}
Assume that $\rho$ is a radially symmetric function such that $\rho\in L^s(\RN) \cap L^\s(B_\d(0))$, for some $s\ge 1$, $\s\ge N$ and $\d>0$. Then the weak solution $\phi_\rho$ of \eqref{eq:BI} is $C^1(\RN;\R)$.
\end{theorem}

We start from the following proposition that basically states that $|\phi_\rho'(r)|=1$ can only happen at $r=0$.
\begin{proposition}\label{pr:C}
Assume that $\rho\in L^s(\RN)$ with $s\ge 1$ and $\rho$ is radially symmetric. Then, $\phi_\rho\in C^1(0,+\infty)$. Moreover, for every $0<r_0<R$ there exists $\eps>0$ such that $|\phi_\rho'(r)| \le 1-\eps$, 
for every $r\in [r_0,R]$.
\end{proposition}

\begin{proof}
From the weak formulation written in radial coordinates, we have
\begin{equation}\label{eqrad}
\int_{0}^{+\infty}\frac{\phi_\rho'(r) r^{N-1}}{\sqrt{1-|\phi_\rho'(r)|^2}}\psi'(r)\, dr
=\int_{0}^{+\infty} r^{N-1}\rho(r)\psi(r)\, dr, 
\end{equation} 
for every $\psi\in\X_{rad}$.
We claim that for any $R>0$, the function 
\beq\label{COR}
h_{R}(r):=\frac{\phi_\rho'(r) r^{N-1}}{\sqrt{1-|\phi_\rho'(r)|^2}}
\eeq
is continuous on $[0,R]$. From this claim we get the regularity of $\phi_\rho$ and we deduce that 
for every $R>0$, there exists a constant $C=C(R)>0$ such that, for every $r\in (0,R]$
\[
\frac{|\phi_\rho'(r)| }{\sqrt{1-|\phi_\rho'(r)|^2}} \le \frac{C(R)}{r^{N-1}}.
\]
Consequently, for every $r_0\in (0,R]$, there exists $\eps>0$ such that
\[
|\phi_\rho'(r)| \le 1-\eps,
\]
for every $r\in [r_0,R]$ .

\medbreak

\noindent To prove our claim, observe that as $\rho\in L^s(\RN)$ for some $s\ge 1$, we deduce from \eqref{eqrad}, that
\[
\left(\frac{\phi_\rho'(r) r^{N-1}}{\sqrt{1-|\phi_\rho'(r)|^2}}\right)'\in L^1(0,R),
\]
for any $R>0$.
Moreover, by \eqref{eq:L1}, we have
\begin{align*}
\int_0^R \frac{|\phi_\rho'(r)| r^{N-1}}{\sqrt{1-|\phi_\rho'(r)|^2}}\,  dr
&=\int_{[0,R]\cap \{|\phi_\rho'(r)|\le \frac 12\}} \frac{|\phi_\rho'(r)| r^{N-1}}{\sqrt{1-|\phi_\rho'(r)|^2}}\, dr
+\int_{[0,R]\cap \{|\phi_\rho'(r)|> \frac 12\}} \frac{|\phi_\rho'(r)| r^{N-1}}{\sqrt{1-|\phi_\rho'(r)|^2}}\, dr
\\
&\le  C\left(\int_0^R r^{N-1} dr + \int_0^R \frac{|\phi_\rho'(r)|^2 r^{N-1}}{\sqrt{1-|\phi_\rho'(r)|^2}} \, dr\right) <+\infty.
\end{align*}
It follows that $h_{R}\in W^{1,1}(0,R)$ and the claim is therefore a consequence of the continuous embedding of $W^{1,1}(0,R)$ in $C([0,R])$.
\end{proof}

Assuming moreover that $\rho \in L^\s(B_\d(0))$ for some $\s\ge N$ and $\d>0$, we can improve the regularity up to the origin.

\begin{proof}[Proof of Theorem \ref{thm:reg}]
Let $R>0$ and $\psi_R:\RN\to \R$ be defined as
\[
\psi_R(x)=
\begin{cases}
R-|x| &|x|\le R,
\\
0 & |x|>R.
\end{cases}
\]
By \eqref{eq:weakBI}, we get
\beq\label{eq:cacca}
-\frac{1}{R} \int_0^R  \frac{\phi_\rho'(r) r^{N-1}}{\sqrt{1-|\phi_\rho'(r)|^2}}\, dr
=\int_0^R \rho(r) r^{N-1}\, dr - \frac{1}{R} \int_0^R \rho(r) r^{N}\, dr.
\eeq
By the Mean Value Theorem (remember that $h_{R}$ defined in \eqref{COR} is continuous), we have
\[
\lim_{R\to 0^+}-\frac{1}{R} \int_0^R  \frac{\phi_\rho'(r) r^{N-1}}{\sqrt{1-|\phi_\rho'(r)|^2}}\, dr
=-\lim_{R\to 0^+} \frac{\phi_\rho'(R) R^{N-1}}{\sqrt{1-|\phi_\rho'(R)|^2}} =c_0\in\R.
\]
Moreover, if $R$ tends to $0$, the right hand side of \eqref{eq:cacca} tends to $0$ by the absolute continuity of the integral.
Therefore $c_0=0$.\\
The equation \eqref{eqrad} means that the weak derivative of $h_{R}(r)$ is equal to $-\rho(r) r^{N-1}$. Since $\rho$ is $L^{1}_{\rm loc}$, we can integrate $h_{R}'$ on $[0,R]$ to get 

\[
\frac{\phi_\rho'(R) R^{N-1}}{\sqrt{1-|\phi_\rho'(R)|^2}}=-\int_0^R \rho(r) r^{N-1} \, dr,
\]
since $c_{0}=0$.
Then, if $R<\delta$,
\[
\frac{ |\phi_\rho'(R)|}{\sqrt{1-|\phi_\rho'(R)|^2}}\le \frac{1}{R^{N-1}}\int_0^R |\rho(r)| r^{N-1} \, dr
\le C R^{(\s-N)/\s} \|\rho\|_{L^\s(B_R(0))}. 
\]
Since $\s\ge N$, we deduce again from the absolute continuity of the integral that $\phi_\rho'(R)\to 0$, as $R\to 0^+$. Hence $\phi_\rho\in C^1([0,\infty))$ with $\phi_\rho'(0)=0$. We conclude therefore that $\phi_\rho\in C^1(\RN;\R)$.
\end{proof}

\subsection{The case of a bounded charge}\label{se:cbg}

Here, keeping in mind Definition \ref{def:spacelike}, we prove Theorem \ref{th:limitato}, whose assumption is $\rho\in L^\infty_{\rm loc}(\mathbb{R}^N)\cap\X^*$.
Let $\Omega$ be an arbitrary bounded domain with smooth boundary in $\mathbb{R}^N$. We set
\[
C_{\phi_\rho}^{0,1}(\Omega)=
\left\{\phi\in C^{0,1}(\Omega) \mid \phi|_{\partial \Omega}=\phi_\rho|_{\partial \Omega}, |\nabla \phi| \le 1\right\}, \]
\begin{equation}
\label{eq:setK}
K=\left\{\overline{xy}\subset\Omega\mid x,y\in\partial\Omega, x\neq y, |\phi_\rho(x)-\phi_\rho(y)|=|x-y|\right\},
\end{equation}
and define $I_\Omega:C_{\phi_\rho}^{0,1}(\Omega)\to\mathbb{R}$ by
\[
I_\Omega(\phi)=\int_{\Omega} \Big(1-\sqrt{1-|\nabla \phi|^2}\Big) \ dx- \int_{\Omega} \rho \phi \ dx .
\]

\begin{proof}[Proof of Theorem \ref{th:limitato}]
Using the above  notations, for any fixed $\O$,  it is easy to see that $\phi_\rho|_{\Omega}$ is a minimizer for $I_\Omega$ in $C_{\phi_\rho}^{0,1}(\Omega)$. 
By \cite[Corollary 4.2]{BS}, we have that $\phi_\rho$ is strictly spacelike in $\Omega\setminus K$ and $Q^-(\phi_\rho)=\rho$ in $\Omega\setminus K$,
where $Q^-$ is defined in \eqref{Q-}. Furthermore,
\[
\phi_\rho(tx+(1-t)y)=t\phi_\rho(x)+(1-t)\phi_\rho(y),
\qquad
0<t<1
\]
for every $x,y\in\partial\Omega$ such that $|\phi_\rho(x)-\phi_\rho(y)|=|x-y|$ and $\overline{xy}\subset\Omega$.
If $K=\emptyset$, then $\phi_\rho$ is strictly spacelike in $\Omega$.
\\
Assume by contradiction that $K\neq\emptyset$. Then there exist $x,y\in\partial\Omega$ such that $x\neq y$, $\overline{xy}\subset\Omega$ and $|\phi_\rho(x)-\phi_\rho(y)|=|x-y|$.
Without loss of generality we can assume that $\phi_\rho(x)>\phi_\rho(y)$. It is easy to see that for all $t\in (0,1)$
\begin{equation}\label{eq:ray}
\phi_\rho (tx +(1-t)y)=\phi_\rho(y)+t|x-y|.
\end{equation}
Since, for any $R>0$ such that $\Omega\subset B_R$, ${\phi_\rho}|_{B_R} $ is a minimizer of $I_{B_R}$ in $C_{\phi_\rho}^{0,1}(B_R)$, then, by \cite[Theorem 3.2]{BS}, we have that \eqref{eq:ray} holds for all $t\in\mathbb{R}$ such that $tx +(1-t)y\in B_R$. Now we reach a contradiction with the boundedness of $\phi_\rho$, for an $R$ sufficiently large.
\end{proof}

\begin{remark}
As observed in \cite[Remark p. 147]{BS}, if $\rho\in C^k(\RN)$, then $\phi_\rho\in C^{k+1}(\RN)$.
\end{remark}

\section{The electric potential produced by a distribution of $k$ point charges}\label{se:delte}

This section is devoted to the proof of Theorem \ref{th:deltafinale}. In all this section we set 
$$\rho=\sum_{i=1}^k a_i \d_{x_i},$$ 
where $a_i\in \R$ and $x_i\in \RN$, for $i=1,\ldots,k$, $k\in\N_{0}$. We consider the problem
\begin{equation}\label{eq:k}
	\left\{
	\begin{array}{ll}
	-\dv \left( \dfrac{\n \phi}{\sqrt{1-|\n \phi|^2}}\right)=\displaystyle\sum_{i=1}^k a_i \d_{x_i}, & \hbox{in }\RN,
	\\
	\phi(x)\to 0, &\hbox{as }x\to \infty.
	\end{array}
	\right.
 \end{equation}

\medbreak

In the recent contribution \cite{K}, the author claims the existence \cite[Proposition 2.1]{K} of a weak solution $v_{\infty}$ of \eqref{eq:k} under the assumption that $x_{i}\in\RT$ and $a_{i}\in\R$ for $i=1,\ldots,k$. However, the proof of \cite[Step 2.6, page 515]{K} is incomplete because even if the Lebesgue measure of the set $\{x\in\R^{3}\mid |\nabla v_{\infty}(x)|=1\}$ is zero, yet one cannot compute the variation $\mathcal F^{(1)}[v_{\infty}](\psi)$ for all test functions $\psi$. Therefore, \cite[(22) page 516]{K} does only vanish for a restricted set of test functions which is not necessarily a dense set of $C_c^\infty (\RN)$. This means that one cannot conclude that $v_{\infty}$ weakly solves \eqref{eq:k}.

\medbreak

The existence of a unique minimizer of the associated energy functional
$$
I(\phi)=\irn \Big(1 - \sqrt{1-|\nabla \phi|^2}\Big) dx
- \sum_{i=1}^{k}a_{i}\phi(x_{i}),
$$
is given in \cite{K} or 
follows  from Theorem \ref{thm:relaxedBI}. As before, we denote this minimizer by $\phi_\rho$. Our concern in this section consists in proving that this minimizer solves \eqref{eq:k} in a weak or a strong sense. We can prove this fact in some particular cases only. 

\medbreak

We need some intermediate steps. First we prove that $\phi_\rho$ satisfies strongly \eqref{eq:k} in $\RN\setminus\G$, where 
$$ \G=\bigcup_{ i\neq j}\overline{x_i x_j}$$ 
and this is true whithout any restriction on the coefficients $a_{i}$ and the location of the charges. This argument is already included in \cite{K}. We give it for completeness. 

\begin{lemma}\label{lem:delta2}
The minimum $\phi_\rho$ of $I$ satisfies strongly
\begin{equation*}
\left\{
\begin{array}{ll}
-\dv \left( \dfrac{\n \phi}{\sqrt{1-|\n \phi|^2}}\right)=0, & \hbox{in }\RN\setminus\G,
\\
\phi(x)\to 0, &\hbox{as }x\to \infty.
\end{array}
\right.
\end{equation*}
Furthermore, we have that 
\begin{enumerate}[label=(\roman*),ref=\roman*]
\item \label{it:411}$\phi_\rho\in C^\infty (\RN\setminus\G)\cap C(\RN)$;
\item \label{it:412}$\phi_\rho$ is strictly spacelike on $\RN\setminus\G$;
\item \label{it:413}for $i\ne j$, either $\phi_{\rho}$ is a classical solution on $\overline{x_i x_j}$, or 
\[
\phi_\rho(tx_{i}+(1-t)x_{j})=t\phi_\rho(x_{i})+(1-t)\phi_\rho(x_{j}),
\qquad
0<t<1.
\]
\end{enumerate}
\end{lemma}

\begin{proof}
Let $\Omega$ be an arbitrary bounded open domain with smooth boundary in $\RN\setminus\{x_1,\ldots,x_k\}$.
Here we repeat the same arguments of Subsection \ref{se:cbg}, using the same notations. The main difference now is that $\rho=0$, since $\O\subset\RN\setminus\{x_1,\ldots,x_k\}$.
As in Subsection \ref{se:cbg}, we infer that $\phi_\rho$ is strictly spacelike and $Q^-(\phi_\rho)=0$ in $\Omega\setminus K$ where $K$ is defined in \eqref{eq:setK}. Furthermore, we have
\[
\phi_\rho(t x+(1-t)y)=t\phi_\rho(x)+(1-t)\phi_\rho(y),
\]
for every $0<t<1$, where $x,y\in\partial\Omega$ are such that $|\phi_\rho(x)-\phi_\rho(y)|=|x-y|$ and $\overline{xy}\subset\Omega$. Again, if $K=\emptyset$, then $\phi_\rho$ is strictly spacelike. \\
We now show that $K$ contains at most $\Gamma$.
Assume by contradiction there exist $x,y\in\partial\Omega$ such that $x\neq y$, $\overline{xy}\subset\Omega$ and $|\phi_\rho(x)-\phi_\rho(y)|=|x-y|$ and such that the straight line spanned by $\overline{xy}$ intersects $\G$ at a finite number of points (possibly zero). Without loss of generality, we can assume that $\phi_\rho(x)>\phi_\rho(y)$. It is easy to see, again, that for all $t\in (0,1)$
\begin{equation}\label{eq:ray2}
\phi_\rho (tx +(1-t)y)=\phi_\rho(y)+t|x-y|.
\end{equation}
Observe also that, since the line spanned by $x$ and $y$ intersects $\G$ at a finite number of points only, we can arbitrarily stretch $\Omega$ in at least one direction of $\overline{xy}$ to build new open sets $\Omega'\subset\RN\setminus\{x_1,\ldots,x_k\}$ with smooth boundaries and such that $\O\subset \O'$. Observe that ${\phi_\rho}|_{\O'} $ is a minimizer for $I_{\O'}$ in $C_{\phi_\rho}^{0,1}(\O')$ 
and, by \cite[Theorem~3.2]{BS}, we have that \eqref{eq:ray2} holds for all $t\in\mathbb{R}$ such that $tx +(1-t)y\in \O'$. 
Now, we reach a contradiction with the boundedness of $\phi_\rho$ by choosing $\O'$ long enough in the direction $\overline{xy}$.\\
Arguing in a similar way, we see that on each edge of $\Gamma$, either $Q^-(\phi_\rho)=0$ or the full edge belongs to $K$, namely (\ref{it:413}) holds.\\
Assertion (\ref{it:411}) follows from \cite[Remark p. 147]{BS}, assertion (\ref{it:412}) from \cite[Corollary 4.2]{BS}. 
\end{proof}

The next Lemma shows somehow the continuity of the minimizer with respect to the coefficients $a_{i}$, $i=1,\ldots,k$. As a consequence, when the coefficients are small, the minimizer $\phi_{\rho}$ is smooth on $\Gamma$.

\begin{lemma}\label{le:delta1}
For any $\eps>0$ there exists $\s>0$ such that, if  $\max_{ i=1,\ldots,k}|a_i|<\s$, then $\|\phi_\rho\|_\infty<\eps$.
\end{lemma}

\begin{proof}
Assume by contradiction that there exists $c>0$ such that, for all $n\in\mathbb{N}$, $n\ge 1$, there exist  $a_1^n,\ldots, a_k^n\in \R$ with $\max_{ i=1,\ldots,k}|a_i^n|<1/n$ and such that 
$\|\phi_{\rho_n}\|_\infty>c$, where $\rho_n=\sum_{i=1}^k a_i^n \d_{x_i}$ and $\phi_{\rho_n}$ is the minimum of the functional associated with $\rho_n$. Since $\rho_n \to 0$ in $\X^*$ and $-\sum_{i=1}^k  \d_{x_i}\le \rho_n\le \sum_{i=1}^k  \d_{x_i}$,  by Theorem \ref{th:convergenza}, we infer that $\phi_{\rho_n}\to \phi_0=0$ uniformly in $\mathbb{R}^N$ and we reach a contradiction.
\end{proof}

Finally, we show that the minimizer is bounded uniformly with respect to the coefficients and that the location $x_{i}$, $i=1,\ldots,k$, do not influence this bound. Again, as a consequence, the minimizer $\phi_{\rho}$  is smooth on $\Gamma$ as soon as the charges are far from each other.

\begin{lemma}\label{le:delta2}
There exists $C=C(a_1,\ldots,a_k)>0$ such that, for all $x_i\in \RN$, $i=1,\ldots,k$, $\|\phi_\rho\|_\infty<C$.
\end{lemma}

\begin{proof}
	Since $I(\phi_\rho)\le 0$ and using \eqref{ineq} and the continuous embedding of $\X$ into $L^\infty(\RN)$, we have
	\[
	c\|\phi_\rho\|_\infty^2
	\le c\|\n \phi_\rho\|_2^2
	\le \irn \Big(1 -\sqrt{1-|\n \phi_\rho|^2}\Big) dx \le \langle \rho,\phi_\rho \rangle
	\le C(a_1,\ldots,a_k) \|\phi_\rho\|_\infty.
	\]
\end{proof}

With these lemmas, we can complete the proof of Theorem \ref{th:deltafinale}. 
\begin{proof}[Proof of Theorem \ref{th:deltafinale}]
We start with the key argument to prove that $\phi_{\rho}$ is a weak solution outside $\{x_1,\ldots,x_k\}$.

\medbreak

\noindent{\it Claim 1: for every bounded domain $\Omega$ such that $\bar{\Omega}\subset\mathbb{R}^N\setminus\{x_1,\ldots,x_k\}$, there exists a unique distributional solution $\bar{\phi}_\rho$ of the problem
	\[
	\left\{
	\begin{array}{ll}
	-\dv \left( \dfrac{\n \phi}{\sqrt{1-|\n \phi_\rho|^2}}\right)=0, & \hbox{in }\O,
	\\
	\phi=\phi_\rho, & \hbox{on }\de \O.
	\end{array}
	\right.
	\]
	}
The proof of this claim follows from a general result of Trudinger on divergence elliptic operators with measurable coefficients, see \cite[Theorem 3.2]{Trud}. The assumption \cite[(3.21)]{Trud} is clearly satisfied in our setting since the right-hand side of the equation is zero in $\Omega$.

\medbreak
	 
\noindent{\it Claim 2: we have $\bar{\phi}_\rho=\phi_\rho$.} By the arbitrariness of $\Omega$ in $\mathbb{R}^N\setminus\{x_1,\ldots,x_k\}$, we deduce that $\bar{\phi}_\rho$ is the unique distributional solution of
	\[
	\left\{
	\begin{array}{ll}
	-\dv \left( \dfrac{\n \phi}{\sqrt{1-|\n \phi_\rho|^2}}\right)=0, & \hbox{in }\RN\setminus\{x_1,\ldots,x_k\},
	\\
	\phi(x)\to 0, &\hbox{as }x\to \infty.
	\end{array}
	\right.
	\]
Lemma \ref{lem:delta2} and the uniqueness of $\bar{\phi}_\rho$ then imply that $\bar{\phi}_\rho=\phi_\rho$ a.e. in $\RN$.

\medbreak

\noindent{\it Proof of (\ref{it:t161}).} Using the same arguments as in the proof of Lemma \ref{lem:delta2}, we aim to prove that $K=\emptyset$.
Assume by contradiction that $K\neq\emptyset$. Then there exist $x,y\in\partial\Omega$ such that $x\neq y$, $\overline{xy}\subset\Omega$ and $|\phi_\rho(x)-\phi_\rho(y)|=|x-y|$.
Without loss of generality, we can assume that $\phi_\rho(x)>\phi_\rho(y)$. It is easy to see that for all $t\in (0,1)$
\begin{equation*} 
\phi_\rho (tx +(1-t)y)=\phi_\rho(y)+t|x-y|.
\end{equation*}
Two possibilities occur: either $\overline{xy}$ intersects $\G$ in a finite number of points (possibly zero), or  $\overline{xy}$  intersects $\G$ in an infinite number of points. In the first case, we conclude as before. In the second case, without loss of generality and applying, if necessary, again \cite[Theorem~3.2]{BS}, we can assume that $\overline{xy}$ can be any piece of $\overline{x_1 x_2}$. Fixing $\eps>0$ such that $2 \eps<\min_{ i,j=1,\ldots,k, \ i\neq j}|x_i-x_j|$, by Lemma \ref{le:delta1}, there exists  $\s>0$ such that, if  $\max_{ i=1,\ldots,k}|a_i|<\s$, then $\|\phi_\rho\|_\infty<\eps$. Since we can find $x',y'\in \overline{x_1x_2}$ with $|x'-y'|>2 \eps$, we reach a contradiction, indeed
\[
2 \eps<|x'-y'|=|\phi_\rho(x')-\phi_\rho(y')|<2 \eps.
\]
The behavior of the gradient of $\phi_\rho$ near the singularities $x_i$ is a consequence of \cite[Theorem 1.5]{E} (see also \cite[Theorem 1.4]{K}).

\medbreak

\noindent{\it Proof of (\ref{it:t162}).} This simply follows by modifying the arguments used to prove (\ref{it:t161}) and taking $\tau=2 C$, where $C$ is given by Lemma \ref{le:delta2}.
\end{proof}

\begin{remark}\label{rem:condTrud}
Since $\rho$ is the divergence of the field 
\[
	F(x) = \sum_{i=1}^{k} b_i \frac{x-x_i}{|x-x_i|^{N}}, 
\]
where $b_i=\frac{a_i}{(N-2)|\mathbb{S}^{N-1}|}$, the equation \eqref{eq:k} can be written 
$$-\dv \left( \dfrac{\n \phi}{\sqrt{1-|\n \phi_\rho|^2}}\right)=\dv F(x),\quad\text{ for } x\in\mathbb{R}^{N}.$$
Trudinger's result  \cite[Theorem 3.2]{Trud} then applies in a bounded open set $\Omega$ containing all the points $x_i$ if 
\begin{equation*}
	\left|\sum_{i=1}^{k} \sum_{j=1}^{k} b_i b_j \int_\O\frac{(x-x_i)\cdot(x-x_j)}{|x-x_i|^{N}|x-x_j|^{N}}\sqrt{1-|\n \phi_\rho|^2} \ dx\right|<+\infty.
	\end{equation*}	
But checking this last assumption is delicate as it requires a precise study of the behaviour of $|\n \phi_\rho|$ around the points $x_{i}$. 
\end{remark}

\medbreak

We conclude this section by showing that we can prove that $\phi_\rho$ satisfies classically \eqref{eq:k} in $\RN\setminus\{x_1,\ldots,x_k\}$  under some symmetry assumptions. This is yet another case where indeed $\phi_{\rho}$ solves the partial differential equation outside the points $\{x_1,\ldots,x_k\}$. 
Assume for simplicity that we have two charges with equal coefficients, namely 
$$\rho =a (\delta_{x_{1}}+\delta_{x_{2}}).$$ By uniqueness of the minimizer and since the functional is now invariant under the orthogonal transformations that exchanges $x_{1}$ and $x_{2}$, we infer that $\phi_{\rho}$ is symmetric and therefore cannot be affine with slope $1$ on the segment $\overline{x_1x_2}$. Therefore, the assertion (\ref{it:413}) of  Lemma \ref{lem:delta2} allows to conclude. The same argument can be used when we have a symmetric configuration of charges with equal coefficient. 

In some special situations, we can argue without assuming any symmetry to prove that the minimizer is not affine with slope $1$ on some of the edges of $\Gamma$. 
As an example, take three charges located at $x_{1}$, $x_{2}$ and $x_{3}$ and suppose that those points are not colinear. One can order the value of $\phi_{\rho}$ and assume without loss of generality that $\phi_{\rho}(x_{1})\le\phi_{\rho}(x_{2})\le\phi_{\rho}(x_{3})$. Then, $\phi_{\rho}$ cannot be affine with slope $1$ on $\overline{x_1x_2}$ and $\overline{x_2x_3}$ since otherwise we have
$$ 
\phi_{\rho}(x_{3})-\phi_{\rho}(x_{1}) = \phi_{\rho}(x_{3})-\phi_{\rho}(x_{2}) + \phi_{\rho}(x_{2})-\phi_{\rho}(x_{1}) = |x_{3}-x_{2}| + |x_{2}-x_{1}|> |x_{3}-x_{1}|.
$$
Other similar situations can be ruled out with the same argument but this is clearly an incomplete and unsatisfactory approach towards the understanding of the general case.

\section{Approximations of the minimizer}\label{se:approx}

In this section, we provide several ways to approximate the minimizer $\phi_{\rho}$ by a sequence of solutions of some approximating PDEs. It is unfortunately unclear that the sequence of PDEs leads to \eqref{eq:BI} at the limit except if $\rho$ is smooth but this case does not require any approximation procedure since it is covered by Theorem \ref{th:limitato}.

\subsection{Approximation through a finite order expansion of the Lagrangian}\label{sse:fop}

One way to overcome, in some sense, the non differentiability of the functional \eqref{eq:functionalBI} was proposed by Fortunato, Orsina and Pisani \cite{FOP} where the authors consider $N=3$ and observe that for $b$ large,
\begin{equation}
\label{approx2}
\mathcal{L}_{\rm BI} = b^2\left(1-\sqrt{1-\frac{|\n \phi|^2}{b^2}}\right) \sim \frac{|\n \phi |^2}{2}   + \frac{|\n \phi |^4}{8b^2}.
\end{equation}
Then, for every density $\rho\in L^1(\RT)$, the Euler equation 
\[
-\dv\left(\left(1+\frac{1}{2b^{2}}|\n \phi|^2\right)\n \phi\right)=\rho, \quad\hbox{in }\RT,
\]
has an unique finite energy solution \cite{FOP}. This means somehow that if we substitute Maxwell's Lagragian by the right-hand side of \eqref{approx2}, the contradiction to the principle of finiteness of the energy disappears. 

We extend the study of approximated solutions \cite{FOP}, see also \cite{Wang}, by looking for higher order expansions, assuming $\rho$ is in the dual space of $D^{1,2}(\RN)\cap D^{1,2n}(\RN)$ for some $n\ge 1$. This include the case of a Radon measure. 
Setting $b=1$, for $n\ge 1$, we define 
$\X_{2n}$ as the completion of $C_c^\infty (\RN)$ with respect to the norm defined by 
\begin{equation*} 
 \|\phi\|^{2}_{\X_{2n}}:=\int_{\RN}|\nabla\phi|^2 + \left(\int_{\RN}|\nabla\phi|^{2n} dx\right)^{1/n}.
\end{equation*}
Formally, the operator $Q^{-}$ (defined in \eqref{Q-}) can be expended as a sum of $2h$-Laplacian, namely
$$Q^{-}(\phi) = -\sum_{h=1}^{\infty} \alpha_{h}\Delta_{2h}\phi,$$
where for all $h\ge 1$, $\alpha_{h}>0$  (the exact values of the coefficient $\alpha_{h}$ are given in \cite{Wang,K}, they are not important for our purpose) and $\Delta_{2h}\phi := \dv(|\nabla \phi|^{2h-2}\nabla \phi)$. The curvature operator is formally the Gateaux derivative of the functional 
$$\irn\left(1 - \sqrt{1-|\nabla \phi|^2}\right) dx = \irn\sum_{h=1}^{\infty}\frac{\alpha_{h}}{2h}|\nabla \phi|^{2h}\, dx,$$
where the power series in the right hand side converges pointwise when $|\nabla \phi(x)|\le 1$.  Assuming $\rho\in\X_{2n}^{*}$, let us denote the $n$th approximation of the functional \eqref{eq:functionalBI} by 
\begin{equation*} 
I_{n}:=\phi\in \X_{2n}  \mapsto  \sum_{h=1}^{n}\frac{\alpha_{h}}{2h} \irn|\nabla \phi|^{2h}\,dx
- \langle \rho, \phi\rangle_{\X_{2n}}.
\end{equation*}
This functional is $C^{1}$ and we have existence and uniqueness of a critical point.

\begin{proposition}\label{pr:approx}
Given $n_0\ge 1$ and $\rho\in \X_{2n_0}^{*}$, then, for all $n\ge n_0$, the functional $I_{n}:\X_{2n}\to\R$ has one and only one critical point. 
\end{proposition}
\begin{proof}
The proof is standard and does not deserve many details. Existence follows from the direct method of the calculus of variation since this functional is bounded from below, coercive
and weakly lower semicontinuous on $\X_{2n}$. Uniqueness follows from the strict convexity of $I_{n}$.
\end{proof}

We now describe the densities covered by this statement. As soon as $2n>\max\{N,2^*\}$, Sobolev inequality combined with Morrey inequality show that $\X_{2n}$ is continuously imbedded in $C^{0,\beta_{n}}_{0}(\RN)$, with $\beta_{n}=1-\frac{N}{2n}$, where we recall that $u\in C^{0,\beta_{n}}_{0}(\RN)$ if there exists $C>0$ such that for every $x,y\in\RN$,
$$|u(x)-u(y)|\le C|x-y|^{\beta_{n}}\quad \text{ and }\ \lim_{|x|\to\infty}u(x)=0.$$	
It follows that if $2n>\max\{N,2^*\}$, $\X_{2n}\subset C_{0}(\RN)$ and therefore $(C_{0}(\RN))^{*}\subset \X_{2n}^*$.

As first important examples, we cover the case $\rho\in L^{1}(\RN)$ since the linear functional 
$$\phi\mapsto \irn \rho\phi\, dx$$
is bounded on $C_{0}(\RN)$ and the case {$\rho = \sum_{i=1}^k a_i \d_{x_i}$ since the linear functional
$$\phi\mapsto\sum_{i=1}^k a_i \phi(x_i)$$
is also bounded on $C_{0}(\RN)$. 
In fact, we can cover the case of Radon measure. Indeed by Riesz-Markov-Kakutani Representation Theorem, see for instance \cite{Fonseca-Leoni}, $(C_{0}(\RN))^{*}$ can be identified with $\mathcal{M}(\RN;\R)$, the space of signed finite Radon measures (i.e. Borel regular measures which are finite on compact sets of $\R^{N}$). This means that if $\rho\in (C_{0}(\RN))^{*}$, there exists a unique Radon measure $\mu$ such that
$$\langle \rho, \phi\rangle = \irn \phi\, d\mu,\quad \text{for all }\phi\in C_{0}(\RN).$$
\\
Observe also that we can also deal with $L^{1}_{\rm loc}(\RN)$ densities. Indeed, the dual spaces of $C_{c}(\RN)$ and $C_{0}(\RN)$ coincide in the sense that if $\rho\in (C_{0}(\RN))^{*}$, then the restriction of $\rho$ to $C_{c}(\RN)$ is a linear bounded functional whereas if $\rho\in (C_{c}(\RN))^{*}$, then it has a unique extension $\bar\rho\in (C_{0}(\RN))^{*}$ such that $\langle \bar\rho, \phi\rangle=\langle \rho, \phi\rangle$ for $\phi\in C_{c}(\RN)$ and the norms of $\rho$ and $\bar \rho$ are equal. 

By interpolation, if $\phi \in \X_{2n}$, then $\phi\in D^{1,q}(\RN)$ for any $2\le q\le 2n$. It follows that $\X_{2n}^{*}$ also contains the weak divergence of any vector field $\xi\in L^{q'}(\RN;\RN)$, where $1/q+1/q'=1$ and $2\le q\le 2n$, i.e. $\frac{2n}{2n-1}\le q'\le 2$. In particular, observe that if $\xi\in L^1(\RN;\RN)$, then $\dv \xi\in \X^{*}$ but in general we cannot conclude that $\dv \xi\in \X_{2n}^{*}$ for a finite integer $n$. This implies that for $\xi\in L^1(\RN;\RN)$, we can take $\rho=\dv \xi$ in Theorem \ref{thm:relaxedBI} but this case is not cover by Proposition \ref{pr:approx} nor by Theorem \ref{th:approxK}. 

\medbreak

Let $\rho \in \X_{2n_0}^*$ for some $n_0\ge 1$ and, for all $n\ge n_0$, let $\phi_n$ be the unique solution of the approximated problem given by Proposition \ref{pr:approx}. The next theorem was basically the heart of the existence of a minimizer of the functional in \cite{K}. We give it here in a general setting, providing a detailed proof for completeness. We mainly follow the idea of \cite[Subsection 2.4 (first part)]{K}.

\begin{theorem}\label{th:approxK}
If $\rho \in \X_{2n_0}^*$ for some $n_0\ge 1$, then $\phi_n$ tends to $\phi_\rho$ weakly in $\X_{2m}$ for all $m\ge n_0$ and uniformly on compact sets.
\end{theorem}

\begin{proof}
Let $\mathcal{I}_n=I_n (\phi_n)$.
For every $m<n$, we have $\I_m \le I_m(\phi_n) < \I_n$.
Thus $(\I_n)_n$ is a strictly increasing sequence.
Since $\I_n<0$ for all $n$, such a sequence is bounded from above.
Hence we infer that $\I:=\lim_{n\to\infty} \I_n\in \R$.\\
Since, for all $m<n$, $I_m(\phi_n)<\I_n<\I$, we deduce that $\|\nabla \phi_n\|_{2m} \le C(m)$ by coercivity.
Then, we infer that $(\phi_n)_n$ is weakly convergent in $\X_{2m}$ 
for all $m\ge  1$.
Hence, by a diagonal argument, we conclude that the limit $\bar{\phi}$ is the same for all $m$ and belongs to the $\bigcap_{m\ge  1} \X_{2m}$.\\
For every $m,n\ge 1$, with $n_0 \le m \le n$, since $\X_{2n} \subset \X_{2n_0}$ and $\I_n<0$, we have
\begin{align*}
\frac{\alpha_1}{2} \|\nabla \phi_n\|_2^2 
+ \frac{\alpha_{n_0}}{2n_0} \|\nabla \phi_n\|_{2n_0}^{2n_0}
 & \le 
\sum_{h=1}^{m} \frac{\alpha_h}{2h} \|\nabla \phi_n\|_{2h}^{2h}\\
& \le 
\sum_{h=1}^{n} \frac{\alpha_h}{2h} \|\nabla \phi_n\|_{2h}^{2h}\\
&\le 
\|\rho\|_{\X_{2n_0}^*} \|\phi_n\|_{\X_{2n_0}}. 
 \end{align*}
It follows that $\|\phi_n\|_{\X_{2n_0}} \le C(n_0)$ and
\[
\sum_{h=1}^{m} \frac{\alpha_h}{2h} \|\nabla \phi_n\|_{2h}^{2h}
\le C(n_0).
\] 
We can now take the limit as $n \to +\infty$ in this inequality. Indeed, $\phi_n\rightharpoonup\bar{\phi}$ in $D^{1,2m}(\RN)$ so that the  weak lower semicontinuity of the norms gives
\[
\sum_{h=1}^{m} \frac{\alpha_h}{2h} \|\nabla \bar\phi\|_{2h}^{2h}
\le C(n_0),
\]
for every $m\ge  n_0$.
This in turn implies that 
\[
\sum_{h=1}^{+\infty} \frac{\alpha_h}{2h} \|\nabla \bar \phi\|_{2h}^{2h} 
\le C
\]
from which we deduce that 
$$\limsup_{h\to \infty}\|\nabla \bar \phi\|_{2h} \le 1.$$
We now claim that $|\nabla \bar{\phi}|\le 1$ a.e. in $\mathbb{R}^N$.
Indeed, assume by contradiction that there exist $\varepsilon>0$ and $\Omega\subset\mathbb{R}^N$ with $|\Omega|\neq 0$ such that $|\nabla \bar{\phi}|\ge  1+\varepsilon$ a.e. in $\Omega$. Then, for every $h\ge 1$
\[
|\Omega|^{1/2h}(1+\varepsilon)
\le 
\left(\int_{\Omega} |\nabla \bar{\phi}|^{2h}dx\right)^{1/2h}
\le 
1
\]
which is a contradiction for $h$ large enough. \\
From here, we argue as in \cite[Subsection 2.5]{K} to prove that $\bar{\phi}=\phi_\rho$. We already know that $\bar{\phi}\in\X$. 
We first show that
\begin{equation}
\label{eq:primaK}
I(\bar{\phi})=\I
\end{equation} 
and then that
\begin{equation}
\label{eq:secondaK}
\I=\min_\X I(\phi).
\end{equation}
To get \eqref{eq:primaK}, observe that, since $\X\subset\X_{2m}$ for every $m\ge 1$, we have $I_m(\bar{\phi})\in\mathbb{R}$ and 
$$\lim_{m\to\infty} I_m (\bar{\phi})=I (\bar{\phi}).$$ 
Moreover $\I_m\le I_m(\bar{\phi})<I (\bar{\phi})$ and so $\I\le I (\bar{\phi})$. On the other hand, since $I_m$ is weakly lower semicontinuous, we have
\[
I_m(\bar{\phi}) \le \lim_{n\to\infty} I_m(\phi_n).
\]
But, since $I_m(\phi_n)<\I$ for $m<n$, we have that $I_m(\bar{\phi}) \le \I$ and so, passing to the limit, we get that $I(\bar{\phi})\le \I$.\\
Finally, to show \eqref{eq:secondaK}, assume by contradiction that there exists $\tilde{\phi}\in\X$ such that $I(\tilde{\phi}) =\I-\varepsilon$ for $\varepsilon>0$. Then $\I_m\le I_m(\tilde{\phi})< \I-\varepsilon$ which contradicts the fact that $\I_m >\I-\varepsilon$ for $m$ large enough.\\
The uniform convergence on compact sets follows arguing as in Lemma \ref{lemma21}.
\end{proof}

We end up this section by observing that other approximation schemes can be used. The truncation of the power series to a finite order gives a lower approximation of the action functional. Another truncation was successfully proposed in \cite{BDD} to deal with a related problem and could have been used here as well. Let us set $a_{0}(s)=(1-s)^{-1/2}$ for all $s<1$. Then 
$$I(\phi) = \frac12  \int_{\mathbb{R}^N} A_{0}(|\nabla \phi|^{2})\, dx - \langle \rho, \phi\rangle,$$
where $A_{0}(t)=\displaystyle\int_{0}^{t}a_{0}(s)\, ds$. 
Take $\theta\in (0,1)$ and define $a_{\theta}:\mathbb{R}^+\to \mathbb{R}^{+}$ by 
\[
a_{\theta}(s) =
\begin{cases}
a_{0}(s) &\mbox{for } s\in[0,1-\theta]\\
\gamma s^{n-1} + \delta & \mbox{for } s\in(1-\theta,+\infty)
\end{cases}
\]
where $\gamma$ and $\delta$ are chosen in such a way that $a_{\theta}$ is $C^{1}$.
The truncated functional $I_{\theta,n}:  \X_{2n}\to\mathbb{R}$ defined by
$$
I_{\theta,n}(\phi):= \frac12  \int_{\mathbb{R}^N} A_{\theta}(|\nabla \phi|^{2})\, dx - \langle \rho, \phi\rangle,
$$
where $A_{\theta}(t)=\displaystyle\int_{0}^{t}a_{\theta}(s)\, ds$ gives another lower estimate of $I(\phi)$. Then we can show that given $n\ge 1$ and $\rho\in \X_{2n}^{*}$, the functional $I_{\theta,n}$ has one and only one critical point which is a weak solution of
$$
-\operatorname{div}\left(a_{\theta}(|\nabla \phi|^2)\nabla \phi\right)=\rho.$$ 
In this approach, $n$ is fixed which makes the functional setting easier than in the finite order approximation of the power series. It is chosen in such a way that $\rho\in \X_{2n}^{*}$. Taking a sequence $\theta_k\to 1$, we can show that the sequence of minimizers of $I_{\theta_k,n}$ converge to the minimizer of $I$, giving yet another way to approach the minimizer by a sequence of solutions of approximating problems. 

\subsection{Approximation by mollification of the charge density}

In the previous section, we have shown that we can approximate the minimizer by a sequence of solutions of a family of approximating problems. We next show that, under sufficient conditions, we can approximate the minimizer by a sequence of solutions of the Born-Infeld equation. The starting point is the fact that a smooth charge yields a smooth minimizer that solves the Euler-Lagrange equation associated to the minimization problem. The approximation is then obtained by a first mollification of the charge.

As before, we still denote by $\phi_\rho$ the minimum associated to $\rho\in \X^*$ and we recall Definition \ref{def:rhopos} that we will use.

\begin{theorem}\label{th:convergenza}
Let $\rho \in \X^*$ and suppose that there exist $(\rho_n)_n \subset \X^*$ and $\tilde{\rho}\in \X^*$ such that $\tilde{\rho}\ge 0$, $\rho_n \to \rho$ in $\X^*$ and $-\tilde{\rho}\le \rho_n\le \tilde{\rho}$. Then $\phi_{\rho_n}$ converges to $\phi_\rho$ weakly in $\X$ and uniformly in $\RN$.
\end{theorem}
\begin{proof}
Observe that by Lemma \ref{monotonia}, $\phi_{\tilde{\rho}}\ge 0$ and $\phi_{-\tilde \rho}\le \phi_{\rho_n}\le \phi_{\tilde{\rho}}$. 
It is easy to see that $\phi_{-\tilde \rho}=-\phi_{\tilde \rho}$ 
and so $|\phi_{\rho_n}|\le \phi_{\tilde \rho}$. By Lemma \ref{compactset}, there exists $\bar \phi\in \X$ such that $\phi_{\rho_n}$ converges to $ \bar \phi$ weakly in $\X$ and uniformly on compact sets. 
This implies that  $|\bar \phi|\le 2 \phi_{\tilde \rho}$. 
Thanks to this uniform decay at infinity, it is easy to see that $\phi_{\rho_n}\to \bar \phi$ uniformly on $\RN$.\\
To show that $\bar \phi=\phi_{\rho}$, let us denote
\[
I_\rho(\phi)=\irn \Big(1-\sqrt{1-|\nabla \phi|^2}\Big)\ dx - \langle\rho, \phi\rangle.
\]
Since $\rho_n \to \rho$ in $\X^*$, $\phi_{\rho_n}\rightharpoonup \bar \phi$  on $\X$ and the first term of $I$ is convex, we infer that
\[
I_\rho(\bar \phi)\le \liminf_{n\to\infty} I_{\rho_n}(\phi_{\rho_n}).
\]
Moreover, since $I_{\rho_n}(\phi_{\rho_n})\le I_{\rho_n}(\phi_\rho)$ for all $n\ge 1$, and
\[
I_\rho(\phi_\rho)=\lim_{n\to\infty} I_{\rho_n}(\phi_\rho),
\]
we have
\[
\limsup_{n\to\infty} I_{\rho_n}(\phi_{\rho_n})
\le \lim_{n\to\infty} I_{\rho_n}(\phi_\rho)
= I_\rho(\phi_\rho)
\]
and we conclude.
\end{proof}

By Theorem \ref{th:limitato} and Theorem \ref{th:convergenza}, we get the following approximation of $\phi_\rho$ as uniform limit of smooth solutions of a sequence of approximated problems. 

\begin{corollary}\label{corollario}
Let $\rho \in \X^*$ and suppose that there exist $(\rho_n)_n \subset \X^*\cap L_{\rm loc}^\infty(\RN)$ and $\tilde{\rho}\in \X^*$ such that $\rho_n \to \rho$ in $\X^*$ and $-\tilde{\rho}\le\rho_n\le \tilde{\rho}$. Then the sequence $(\phi_{\rho_n})_{n}$ of (locally strictly) spacelike solutions of \eqref{eq:BI} with $\rho_{n}$ converges to $ \phi_\rho$ weakly in $\X$ and  uniformly in $\RN$.
\end{corollary}

\begin{remark}
If $(\rho_n)_n\subset L^p(\RN)$, with  $1\le p<+\infty$, is such that $\rho_n \to\rho$ in $L^p(\RN)$, we can immediately conclude that $\phi_{\rho_n} $ converges to $ \phi_\rho$ weakly in $\X$ and uniformly in $\RN$. In particular, for a datum $\rho\in L^p(\RN)$, the approximating sequence $(\phi_{\rho_n})_{n}$, where $(\rho_{n})_{n}$ is a standard sequence of mollifications of $\rho$, is made of smooth strictly spacelike solutions of \eqref{eq:BI} with the data $\rho_{n}$.
\end{remark}

\section{The Born-Infeld-Klein-Gordon equation and other extensions}\label{sec:conclusion}

Another interesting problem which involves the Born-Infeld theory appears when we couple a field, governed by the nonlinear Klein-Gordon equation, with the electromagnetic field whose Lagrangian density is given by \eqref{eq:Bcc} or \eqref{eq:BIcc}, by means of the Weil covariant derivatives.

\medbreak

In the wake of \cite{DP,M}, Yu, in \cite{yu}, deals with the system
\beq\label{eq:yu}
\left\{
\begin{array}{ll}
\displaystyle \operatorname{div}\left(\frac{\nabla \phi}{\sqrt{1-|\nabla \phi|^2}}\right)
=u^2(\o +\phi), & x\in\mathbb{R}^3,
\\
\Delta u=\left(m^2-(\o+\phi)^2\right)u-|u|^{p-2}u, & x\in\mathbb{R}^3.
\end{array}
\right.
\eeq
Fixing $u$ in a convenient space of radial functions, Yu considers the functional 
\[
E_u(\phi)=\int_{\RT} \left[\Big(1 - \sqrt{1-|\nabla \phi|^2}\Big)+\o u^2 \phi +\frac 12 \phi^2 u^2\right]dx,
\]
and proves that $E_u$ possesses a minimizer $\phi_u$ without proving that the minimum $\phi_u$ is a critical point of $E_u$. Then the second equation of \eqref{eq:yu} is solved with $\phi_u$ in place of $\phi$. Yu's conclusion is then that  $(u,\phi_u)$ is a solution of \eqref{eq:yu} in a generalized sense, meaning that the second equation is classically satisfied while $\phi_u$ is a minimizer of $E_{u}$.

Our aim here is to show that the minimizer of $E_u$, and of similar functionals, is actually a solution of the corresponding equation. This leads us to consider more general equations of the form
\begin{equation}\label{eqf}
\left\{
\begin{array}{ll}
-\operatorname{div}\left(\dfrac{\nabla \phi}{\sqrt{1-|\nabla \phi|^2}}\right)
+f(x,\phi)=0 , & x\in\mathbb{R}^N,
\\
\phi(x)\to 0, &\hbox{as }|x|\to \infty.
\end{array}
\right.
\end{equation}
We assume that $f:\RN\times \R \to \R$ is a Carath\'eodory function such that
\begin{enumerate}[label=(f\arabic{*}), ref=f\arabic{*}]
\item \label{it:f1} there exists $p\ge 2^*-1$ such that for all $(x,t)\in \RN\times \R$ 
\[
|f(x,t)|\le C |t|^p;
\] 
\item \label{it:f2} $f(\cdot, t)$ is radially symmetric, for all $t\in \R$;
\item \label{it:f3} the functional $I_F:\X \to \R$ defined by
\[
I_F(\phi)=\irn \Big(1 - \sqrt{1-|\nabla \phi|^2}\Big) \ dx+\irn F(x,\phi)\ dx,
\]
where $F(x,t)=\displaystyle\int_0^t f(x,s)\, ds$, has a nontrivial radial local minimum $\phi_f$ in $\X$.
\end{enumerate}

\begin{remark}
The existence of a local minimum of $I_F$ in $\X$ follows, for example, by standard assumptions such as the coercivity of $I_F$ and the convexity of the function $F(x,\cdot)$.
\end{remark}

We can prove the following

\begin{theorem}\label{th:f}
Suppose that (\ref{it:f1})-(\ref{it:f3}) hold, then $\phi_f$ is a nontrivial weak solution of \eqref{eqf}. 
\end{theorem}
\begin{proof}
Let $\phi_{f}\in\X$ be a local minimum of $I_{F}$. Arguing as in Proposition \ref{pr:sol-deb}, since the map 
$\phi\in \X\mapsto \irn F(x,\phi)$ is of class $C^1$, we infer that
$$\irn \frac{  |\n  \phi_{f}|^2}{ \sqrt{1- |\n \phi_{f}|^2}}\, dx + \irn f(x,\phi_{f})\phi_{f}\, dx \le 0$$
and that the set of points where $|\n  \phi|=1$ has zero measure. Then one concludes as in the proof of Theorem \ref{thm:exsol}. 
\end{proof}

\begin{remark}
Arguing as in Theorem \ref{th:f} we can complete the arguments of \cite{yu} concerning the existence of a nontrivial solution of \eqref{eq:yu}. In that precise case, one can even conclude that the solution is classical and even smooth.
\end{remark}

We finish this section by showing that, if $I_F$ has a nontrivial {\em local minimizer} in the generalized sense of Morse, i.e. the function is minimal with respect to compactly supported variations, see for example \cite{AAC}, then it is a solution of the \eqref{eqf}. Of course, any local or global minimizer is a local minimizer in the sense of Morse. We emphasize that we do not require any radial symmetry here.  
\begin{theorem}\label{th:f2}
Suppose that (\ref{it:f1}) holds and that there exists $\phi_0\in \X$ such that for any bounded open set $\O\subset \RN$, and for any $\phi\in \X$ with $\phi=\phi_0$ in $\RN\setminus \O$, $I_F(\phi_0)\le I_F(\phi)$. Then $\phi_0$ is a weak solution of \eqref{eqf}.
\end{theorem}

\begin{proof}
Set $\rho(x) := f(x,\phi_0(x))$. For any $\O \subset \RN$ bounded, arguing as in the previous sections, we infer that the functional $I_\Omega$ defined by
\[
I_\Omega(\phi)=\int_\Omega \Big(1-\sqrt{1-|\nabla \phi|^2}\Big) \ dx + \int_\Omega \rho\phi \ dx
\]
has a unique minimizer $\psi_\Omega\in\X_\Omega$, where $\X_\Omega$ is the set of functions $\phi\in \X$ with $\phi=\phi_0$ in $\RN\setminus \O$.

\medbreak
\noindent For every $x_{0}\in\RN$ and $R>0$, we simply denote the minimizer $\psi_{B_{2R}(x_{0})}$ by $\psi_{R,0}$ where $B_{2R}(x_{0})$ is the ball of radius $2R$ centered at $x_{0}$. We claim that
\begin{equation}\label{I-hope-it-is-the-last-equation}
-\operatorname{div}\left(\dfrac{\nabla \psi_{R,0}}{\sqrt{1-|\nabla \psi_{R,0}|^2}}\right)
+f(x,\phi_{0})=0
\end{equation}
on the ball $B_R(x_{0})$ when $R>0$ is sufficiently large. Since $\phi_0$ is bounded, we may assume that $|\phi_0(x)-\phi_0(y)|\le R/2$ for every $x,y\in \R^N$ and therefore, as soon as $|x-y|\ge R$, we have  $|\phi_0(x)-\phi_0(y)|\le |x-y|/2$.
Arguing as Subsection \ref{se:cbg}, and keeping the same notations, we know that $\psi_{R,0}$ is strictly spacelike and solve \eqref{I-hope-it-is-the-last-equation} in $B_{2R}(x_{0})\setminus K$ where 
\[
K=\left\{\overline{xy}\subset B_{2R}(x_{0})\mid x,y\in\partial B_{2R}(x_{0}), x\neq y, |\phi_0(x)-\phi_0(y)|=|x-y|\right\}.
\]
Of course, if $K\cap B_R(x_{0})=\emptyset$, then our claim is proved. Assume by contradiction that $K\cap B_R(x_{0})\neq\emptyset$. Then there exist $x,y\in\partial B_{2R}(x_{0})$ such that $x\neq y$, $\overline{xy}\cap B_R(x_{0})\neq \emptyset$ and $|\phi_0(x)-\phi_0(y)|=|x-y|$. But it is easy to see that in such a case $|x-y|\ge R$ and so $|\phi_0(x)-\phi_0(y)|\le |x-y|/2$ which is a contradiction. 

\medbreak
\noindent We now conclude by showing that $\psi_{\Omega}=\phi_0$ whatever $\Omega\subset\RN$ which implies $\psi_{R,0}=\phi_0$ for every $x_{0}\in\RN$ and every $R>0$. This follows from a totally standard argument in convex analysis. To simplify the notations, we set
\[
J_\Omega(\phi)=\int_\Omega \Big(1-\sqrt{1-|\nabla \phi|^2}\Big) \ dx. 
\]
Since $\phi_0$ is a local minimizer in the sense of Morse for $I_F$, we deduce that for $0<t<1$, we have
$$J_{\Omega}((1-t)\psi_{\Omega} + t\phi_{0}) +  \int_{\Omega}F(x,(1-t)\psi_{\Omega}+t\phi_{0})\, dx \ge J_{\Omega}(\phi_{0}) + \int_{\Omega}F(x,\phi_0)\, dx.$$
Using the convexity of $J_{\Omega}$, we deduce that  
$$(1-t)J_\Omega(\psi_{\Omega}) + t J_\Omega(\phi_{0}) + \int_{\Omega}F(x,(1-t)\psi_{\Omega}+t\phi_{0})\, dx \ge J_\Omega(\phi_{0}) + \int_{\Omega}F(x,\phi_0)\, dx,$$
or equivalently
$$J_\Omega(\psi_{\Omega}) \ge J_\Omega(\phi_{0}) + \frac{1}{t-1}\left(\int_{\Omega}F(x,(1-t)\psi_{\Omega}+t\phi_{0})\, dx
-\int_{\Omega}F(x,\phi_0)\, dx\right).$$
This yields
$$J_\Omega(\psi_{\Omega}) \ge J_\Omega(\phi_{0}) + \int_{\Omega}f(x,\phi_0(x))(\phi_0-\psi_{\Omega})$$
and therefore 
$$I_{\Omega} (\psi_{\Omega}) \ge I_{\Omega} (\phi_{0}).$$
By uniqueness of the minimizer of $I_{\Omega}$, we conclude that $\psi_{\Omega} = \phi_{0}$ in $\Omega$.
\end{proof}

\end{document}